\documentclass{amsart}

\usepackage{amssymb,bbm,mathdots,mathrsfs}
\usepackage{hyperref}
\usepackage{enumitem}
\usepackage{graphicx,psfrag,subfigure}
\usepackage{tikz}
\usetikzlibrary{arrows.meta,patterns}

\usepackage{todonotes}

\newtheorem{theorem}{Theorem}[section]
\newtheorem{lemma}[theorem]{Lemma}
\newtheorem{corollary}[theorem]{Corollary}

{\theoremstyle{definition}
\newtheorem{definition}[theorem]{Definition}
\newtheorem{example}[theorem]{Example}
\newtheorem{remark}{Remark}}

\newcommand{\Z}{\mathbb{Z}}
\newcommand{\R}{\mathbb{R}}
\newcommand{\C}{\mathbb{C}}
\newcommand{\ind}{\mathbbm{1}}

\graphicspath{{Fig/}}

\begin{document}

\title[On the real Jacobian conjecture]{A sharp degree bound in the real Jacobian conjecture}

\author[F. Braun, J. Gwo\'zdziewicz, F. Fernandes \MakeLowercase{and} B. Or\'{e}fice-Okamoto]
{Francisco Braun$^{*}$,  Filipe Fernandes$^\dagger$, Janusz Gwo\'zdziewicz$^{\ddagger}$, \MakeLowercase{and} Bruna Or\'{e}fice-Okamoto$^*$}

\address{$^{*}$ Departamento de Matem\'{a}tica, Universidade Federal de S\~ao Carlos, 13565--905 S\~ao Carlos, S\~ao Paulo, Brazil}
\email{franciscobraun@ufscar.br}
\email{brunaorefice@ufscar.br}

\address{$^\dagger$ Universidade do Distrito Federal Jorge Amaury, 70635--815 Bras\'ilia, Distrito Fe\-de\-ral, Brazil}
\email{filipematonb@gmail.com} 

\address{$^{\ddagger}$ Institute of Mathematics\\ 
University of the National Education Commission, Kra\-kow\\
Podchor\c a{\accent95 z}ych 2\\
PL-30-084 Cracow, Poland}
\email{janusz.gwozdziewicz@uken.krakow.pl}



\date{\today}

\begin{abstract}
Let $F=(p,q):\mathbb R^2\to \mathbb R^2$ be a polynomial map with nowhere zero Jacobian determinant. 
A long-standing problem is to determine the largest integer $k$ such that the condition $\deg p\le k$ guarantees the global injectivity of $F$. 
Although several partial results have been obtained over the past $30$ years, the sharp degree bound has remained unknown.
In this paper, we prove that $F$ is injective whenever $\deg p=6$. 
On the other hand, we construct a non-injective polynomial map with nowhere vanishing Jacobian determinant for which $\deg p=7$. 
Combined with the previously known injectivity results for $\deg p\le 5$, our results completely settle the problem and establish the optimal degree bound. More precisely, we show that $7$ is the minimal degree for which non-injective examples can occur.
\end{abstract}

\maketitle

\section*{Introduction}
Let $F=(p,q):\R^2\to\R^2$ be a smooth map such that
\begin{equation}\label{1}
JF(x,y)\neq 0 \ \forall (x,y)\in\R^2, 
\end{equation}
where $JF(x,y)$ stands for the Jacobian determinant of $F$ at $(x,y)$. 
By the Inverse Function Theorem, $F$ is a local diffeomorphism but not necessarily a global one. 

If we assume in addition that $F$ is a polynomial map, the \emph{real Jacobian conjecture} states that~\eqref{1} implies that $F$ is a global diffeomorphism. 
This conjecture is not true: in 1994 Pinchuk \cite{P} constructed a class of non-injective polynomial maps $F = (p,q)$ satisfying~\eqref{1}. 
In this class $p$ is a fixed polynomial of degree $10$ and $q$ varies, but the smallest degree it can have is $25$, see \cite{C} for more details. 

On the other hand, if $p$ has degree $1$, then we can assume $p(x,y) = x$ and $q_y(x,y) \neq 0$, so that $y \mapsto q(x,y)$ is injective for each $x$ and then the real Jacobian conjecture holds for $F$. 
Let $k$ be the positive integer such that the real Jacobian conjecture is true with the additional assumption that the degree of $p$ is less than $k$ and such that there exists a counterexample with $\deg p = k$. 
Then $2 \leq k \leq 10$. 

Since the appearance of the Pinchuk counterexample, researchers have been trying to determine the exact value of $k$. 
It is not difficult to conclude, by extending a little the above reasoning that $k > 2$. 

In \cite{G} it was proved that the real Jacobian conjecture holds if both $\deg p < 4$ and $\deg q < 4$. 
Subsequently, in  \cite{BFOO,BOO,BST}, it was concluded that the conjecture holds by assuming that $\deg p < 6$, see also \cite{G1}. 

On the other hand, in \cite{F} it was constructed a counterexample to the real Jacobian conjecture with $\deg p = 10$ and $\deg q = 15$. 
Then in \cite{BF}, a counterexample with $\deg p = 9$ and $\deg q = 15$ appeared. 

So far it follows then that $6 \leq k \leq 9$. 
In this work, we show that $k = 7$. 
Indeed, Corollary \ref{cmain} shows that $k \geq 7$, and Corollary \ref{ultimo} shows that $k \leq 7$. 

Our results were obtained by combining tools from algebraic geometry and dynamical systems. 
Indeed, in Section \ref{453} we recall general key facts in the context of the global injectivity, briefly explaining our strategy. 
In Section \ref{NewtonP} we recall the tools from algebraic geometry we will employ in our paper. 
It is here that we will introduce the function $N_S(c)$, which counts the number of branches at infinity of the curve $p(x,y) = c$ associated with the outer edge $S$ of the Newton polygon of the polynomial $p$. 
When $p$ is a submersion, the sum of the functions $N_S(c)$, considering all the outer edges, is $N(c)$, which is precisely the  number of connected components of the fiber $p^{-1}(c)$. 
In this section, Lemma \ref{bounded}, with results on $N_S$, will be the main one. 
We finish the section by defining the Euler integral of $N_S$ and recalling S\c{e}kalski's theorem relating the Euler integral with the index of the gradient of $p$ at infinity. 

In Section \ref{ds} we recall the Poincar\'e compactification of polynomial vector fields as well as some key concepts on dynamical systems we will need. 
Then we particularize to the Hamiltonian vector field $\nabla p^{\perp}$ associated with a polynomial function $p$. 
The concept of hyperbolic sectors at infinity of $\nabla p^{\perp}$ will be recalled. 
It is the main subject of Theorem \ref{hyperbolic}, a result from \cite{BFOO}, that gives conditions for the non-existence of Jacobian mates for polynomial submersions. 

In Section \ref{agmds} we put together the techniques of the previous sections to construct the main tools we will use in our proof. 
In particular, we prove Theorem \ref{lattice_points}, that precisely describes the function $N_S(c)$ when $S$ has only one interior lattice point. 
Lemmas \ref{new1} and \ref{lemma2} will relate the existence of Jacobian mates with some properties on edges of the Newton polygon of $p$. 

In Section \ref{mr} we prove that $k \leq 7$ by applying the results of the previous sections. 

Then we end the paper in Section \ref{explicitizing} presenting a non-injective polynomial map $(p,q): \R^2 \to \R^2$ whose Jacobian determinant is nowhere zero such that $\deg p = 7$. 

\section{Candidates for $p$}\label{453}
\subsection{Connectedness of fibers and bifurcation set}\label{45601}
It is well known that a polynomial map $F$ satisfying~\eqref{1} is injective if and only if the fibers of $p$ are connected. 
Indeed,~\eqref{1} ensures the monotonicity of $q$ along the connected components of $p^{-1}(c)$, what proves the ``only if'' part. 
The ``if'' part follows because injective polynomial maps of $\R^n$ are automatically onto, see for instance \cite{BBR}, so the fibers of $p$ are connected. 
Then for $F$ satisfying~\eqref{1}, the injectivity of $F$ is equivalent to $p$ being a \emph{locally trivial fibration}, that is, the \emph{bifurcation set of} $p$, that we will denote by $B(p)$, is empty, see more details, for instance, in \cite{BST, TZ}. 

\subsection{Highest homogeneous part}
For a polynomial function $p$, we will denote by $p_+$ its homogeneous part of highest degree. 
From \cite[Lemma 4.2]{BFOO}, as well as from \cite[Theorem 2.2]{CGM} and \cite[Proposition 4]{BV}, it follows that if a polynomial submersion has a disconnected fiber then $p_+$ must have linear factors of multiplicity at least $2$. 
Therefore, since we want to prove the injectivity of maps $F$ satisfying~\eqref{1}, we only need to be concerned with polynomial submersions $p$ such that 
\begin{equation}\label{multiplefactor}
p_+(x,y) = (a x + b y)^2 r(x,y), 
\end{equation}
where $a,b \in \R$ and $r(x,y)$ is a homogeneous polynomial. 

\subsection{Jacobian mates}
We say that a polynomial submersion $p: \R^2 \to \R$ has a \emph{Jacobian mate} when there exists a polynomial submersion $q: \R^2 \to \R$ such that the map $(p,q)$ satisfies~\eqref{1}. 
In this case, $q$ is a \emph{Jacobian mate of $p$}. 

Our strategy is to consider polynomial submersions of degree $6$ with disconnected fibers or, equivalently, such that $B(p) \neq \emptyset$, and to prove that they have no Jacobian mates. 

\subsection{Quadratic-like}\label{ql}
It was proved in \cite{BFM} that quadratic-like polynomial submersions -- those of degree less than $3$ in $x$ or $y$ -- either have all fibers connected or have no Jacobian mates. 
Therefore, we only need to be concerned with the polynomial submersions having degree at least $3$ in both $x$ and in $y$. 

\section{Newton polygon for counting connected components of fibers: the branches at infinity technique}\label{NewtonP}
\subsection{Branches at infinity}
Throughout this paper, any mention to \emph{infinity}, or to the \emph{line at infinity}, will mean the border of $\R^2$ in $\R \mathbb P^2$. 
As usually, we shall represent $\R \mathbb P^2$ by the closed unit disc identifying antipodal points on its boundary. 
In this model, the boundary of the disc corresponds to the line at infinity. 
For any polynomial $p(x,y)$ of degree $d$, by its \emph{homogenization} we mean the polynomial $P(z_1,z_2,z_3) = z_3^d p(z_1/z_3, z_2/z_3)$. 

We will say that an infinity point \emph{$[a : b : 0]$ is contained in the algebraic curve $p(x,y) = 0$} when $P(a, b, 0) = 0$. 
This is equivalent to the real linear factor $a y - b x$ being a divisor of $p_+(x,y)$. 

\subsection{Milnor at infinity}
It is well known, see Milnor \cite[Lemma 3.3]{M2}, that for a non-isolated point $z$ of a given real algebraic curve $p(x,y) = 0$ there exists a neighborhood $U$ of $z$ such that $p^{-1}(0) \cap U$ is the union of finitely many \emph{branches} intersecting only at $z$. 
Each branch is homeomorphic to an interval $(-\epsilon, \epsilon)$ under an analytic parametrization $s(t) = \left(s_1(t), s_2(t)\right)$ with $s(0) = z$. 
The images $s(-\epsilon, 0)$ and $s(0, \epsilon)$ are the \emph{half-branches} of the branch. 
In this case, we say that each of these branches or half-branches is \emph{at $z$}.  

For a \emph{branch at infinity} of an algebraic curve we will mean a branch of the curve at a point in the line of infinity. 

By considering the homogeneous coordinates and their inverses at the chart of $\R \mathbb P^2$ (see also Section \ref{Pc} below), neighborhood of a point in the line of infinity contained in an algebraic curve $p(x,y) = 0$, we can ``convert'' to $\R^2$ the parametrization of any branch at infinity not entirely contained in the coordinate axes at infinity to a Laurent series. 
To be more precise, let $s(t)$ be the above parametrization of a branch of an algebraic curve at infinity with $s_1(t) s_2(t) \not \equiv 0$. 
Deleting the point at infinity, we can put this branch back to $\R^2$, and parametrize it as $\gamma(t) = (x(t), y(t))$, $0 < |t|\leq \epsilon$, where 
\begin{equation}\label{branch}
\begin{aligned}
x(t) & = A t^{-\alpha} + \mbox{ terms of higher degrees, } \\
y(t) & = B t^{-\beta} + \mbox{ terms of higher degrees, } 
\end{aligned}
\end{equation}
for suitable $A, B \in \R$, $A B \neq 0$, and $(\alpha, \beta) \in \Z^2$, with $\alpha>0$ or $\beta>0$. 
It is easy to conclude that this \emph{branch} $\gamma$ is at $[A:B:0]$ if $\alpha=\beta$, at $[1:0:0]$ if $\alpha>\beta$, and at $[0:1:0]$ if $\alpha<\beta$. 

\subsection{Newton polygon}
For a given polynomial $p(x,y) = \sum_{i,j} a_{i j} x^i y^j$, we define its \emph{support} as the subset of $\Z^2$ given by $\{(i,j)\ |\ a_{i j} \neq 0\}$, and its \emph{Newton polygon}, denoted by $NP(p)$, as the convex hull of its support in $\R^2$. 
Up to a linear change of variables, we can always assume that $NP(p)$ is \emph{convenient}, that is, that it touches both the $x$ and $y$ axes away from the origin. 
In this case we also say that \emph{$p$ is convenient}. 
Up to adding a constant, we may assume that a convenient Newton polygon has positive area. 
In this case, for each edge $S$ of a convenient $NP(p)$, we define the vector $\xi \in \Z^2$, with coprime coordinates, that is orthogonal to $S$ and points outward $NP(p)$ as the \emph{outer normal vector of $S$}. 

Any edge such that its outer normal vector has at least one positive coordinate is said to be an \emph{outer edge}. 

For a face $S$ of $NP(p)$ we define the polynomial $p_S(x, y) = \sum_{(i,j) \in S} a_{i,j} x^i y^j$, that can be written as, see \cite{LO}, 
\begin{equation}\label{restriction}
p_S(x,y) = \delta x^r y^s \prod_{i=1}^k (y^{\xi_1} - c_i x^{\xi_2})^{\nu_i}, 
\end{equation}
for suitable $r, s, \nu_i \in \Z$, $\delta \in \R$, and $c_i \in \C$, $i = 1,\ldots, k$. 
The $c_i$, $i = 1,\ldots, k$, are pairwise distinct and non-zero. 

\begin{definition}
The polynomial $p(x,y)$ is \emph{nondegenerate on the edge $S$} of $NP(p)$ if $\nu_i = 1$ whenever $c_i \in \R$, for all $i = 1, \ldots, k$. 
Besides, $p(x,y)$ is \emph{nondegenerate with respect to $NP(p)$}, or simply \emph{nondegenerate}, when $NP(p)$ is convenient and $p(x,y)$ is nondegenerate on each outer edges of $NP(p)$. 
\end{definition}

Sometimes we will  extend the definition of Newton polygon for forms $p(x,y) = \sum_{(i,j) \in \Z^2} a_{i j} x^i y^j$ in the natural way. 

\subsection{Branches associated with outer edges of $NP(p)$}
\begin{definition}\label{mab}
Let $p(x,y)$ be a polynomial. 
We say that a \emph{branch $\gamma$ of $p(x,y) = 0$ at infinity is associated with the outer edge $S$ of $NP(p)$} if there exists a positive integer $m$ such that $(\alpha,\beta) = m \xi$, where $(\alpha, \beta)$ is given in~\eqref{branch} and $\xi$ is the outer normal vector of $S$. 
\end{definition} 

\begin{remark}\label{pAB}
If $p(x,y)$ is convenient, it follows that each branch at infinity of $p(x,y) = 0$ is associated with an outer edge of $NP(p)$. 
Indeed, let $\gamma$ be such a branch parametrized as in~\eqref{branch}, and write  
$$
p = p_\ell + p_{\ell + 1} + \dots + p_d, 
$$
with $\ell < d \in \Z$, the weighted homogeneous decomposition of $p$ with respect to the weights $(\alpha,\beta)$. 
Then
$$
\begin{aligned}
0 = t^d  \sum_{k = \ell}^d p_k\big(x(t), y(t)\big) & = t^d \sum_{k=\ell}^d p_k\Big(t^{-\alpha}(A + \sum_{j = 1}^\infty a_j t^j ),t^{-\beta}(B + \sum_{j = 1}^{\infty} b_j t^j)\Big)\\
&=\sum_{k=\ell}^dt^{d - k}p_k\Big(A + \sum_{j = 1}^\infty a_j t^j , B + \sum_{j = 1}^{\infty} b_j t^j\Big), 
\end{aligned}
$$
for suitable $a_j, b_j \in \R$. 
Letting $t \to 0$ above, we conclude that 
$$
p_d(A,B)=0. 
$$ 
Since $AB\neq 0$, it follows that $NP(p)$ has an outer edge with normal vector parallel to $(\alpha,\beta)$.
\end{remark}

Note that outer edges of a convenient Newton polygon do not change when $p$ is replaced by $p - c$ for any $c \in \R$. 
So it makes sense to define the following. 
\begin{definition}
Let $p(x,y)$ be a convenient polynomial and $S$ be an outer edge of $NP(p)$. 
We define $N_S(c)$ as the number of branches of $p(x,y) - c = 0$ at infinity associated with $S$. 

We further denote by $N(c)$ the number of branches of $p(x,y) - c = 0$ at infinity. 
\end{definition}

\begin{remark}\label{45r}
\begin{enumerate}
\item It is clear that for a convenient $p(x,y)$ it follows that $N(c) = \sum N_S(c)$, where the sum runs over the outer edges of $NP(p)$. 
\item Also, for a polynomial submersion $p(x,y)$, it is clear that $N(c)$ is precisely the number of connected components of the fiber $p^{-1}(c)$. 
So, in this case, the function $N$ agrees with the definition given in \cite{BFOO}. 
\end{enumerate}
\end{remark}

The following lemma provides some general properties of the function $N_S$. 

\begin{lemma}\label{bounded}
Let $p(x,y)$ be a convenient polynomial and let $S$ be an outer edge of $NP(p)$. 
Denote $\overline N_S:=\# (S\cap \mathbb{Z}^2)-1$. 
Then 
\begin{enumerate}[label={\textnormal{(\roman*)}}]
\item the function $N_S$ is bounded from above by $\overline N_S$, 
\item the function $N_S$ is constant $\mod 2$ outside a finite set, 
\item if $p$ is nondegenerate on $S$ then $N_S$ is constant  and $N_S \equiv \overline{N}_S \mod{2}$, 
\item if $S$ has no interior lattice points, then $N_S\equiv1$,
\item if $N_S(c_0)>0$, then there exists $\epsilon >0$ such that the restrictions of $N_S$ to the intervals $(c_0-\epsilon,c_0)$ and $(c_0,c_0+\epsilon)$ are constant and at least one of these functions is nonzero.
\end{enumerate}
\end{lemma}
\begin{proof}
In the proof we will use a specific monomial substitution. 
Its aim is to replace counting branches at infinity with counting local branches of algebraic curves.  

Let $\xi=(\xi_1,\xi_2)$ be the outer normal vector to $S$, and write the polynomial $p_S$ as in~\eqref{restriction}, where, without loss of generality, we may assume that $\delta=1$. 
We may also assume that $p(0,0)\neq0$; in this case $(0,0)$ is a vertex of $NP(p)$.  

Since the components of $\xi$ are coprime integers, there exists $\eta=(\eta_1,\eta_2)\in\Z^2$ such that  $\xi_1\eta_2-\xi_2\eta_1=1$. 
Then the matrix 
$
\begin{pmatrix} 
\eta_1 & -\xi_1 \\  
\eta_2 & -\xi_2 
\end{pmatrix}
$
is invertible, with inverse
$
\begin{pmatrix}
-\xi_2 & \xi_1 \\ 
-\eta_2 & \eta_1 
\end{pmatrix}. 
$
Consequently the rational mapping $\sigma: (\R\setminus\{0\})^2\to (\R\setminus\{0\})^2$ given by 
\begin{equation}\label{toric}
\begin{aligned}
x &=  u^{\eta_1} v^{-\xi_1},\\
y & =  u^{\eta_2} v^{-\xi_2}
\end{aligned} 
\end{equation}
has the inverse $\sigma^{-1}$ given by 
\begin{equation*}
\begin{aligned} 
u & = x^{-\xi_2}y^{\xi_1}, \\
v & = x^{-\eta_2}y^{\eta_1}. 
\end{aligned} 
\end{equation*}

Now we describe the effect of substituting~\eqref{toric} in the polynomial $p$. 
We have 
\begin{equation}\label{monomial}
x^a y^b = (u^{\eta_1} v^{-\xi_1})^a(u^{\eta_2} v^{-\xi_2})^b = u^{\eta_1 a+\eta_2 b}v^{-\xi_1a - \xi_2b}, 
\end{equation}
\begin{equation}\label{binomial}
y^{\xi_1} - c x^{\xi_2} = (u^{\eta_2} v^{-\xi_2})^{\xi_1}-c (u^{\eta_1} v^{-\xi_1})^{\xi_2} = u^{\eta_1\xi_2}v^{-\xi_1\xi_2}(u-c). 
\end{equation}

It follows from~\eqref{monomial} that the Newton polygon of 
$$
\tilde p(u,v):=p(u^{\eta_1} v^{-\xi_1},u^{\eta_2} v^{-\xi_2})
$$ 
is the image of $NP(p)$ by the linear mapping $T:\R^2 \to \R^2$ given by 
$$
T(q) = \big(\langle\eta, q\rangle,-\langle\xi,q\rangle \big). 
$$ 
Let $q_1$, $q_2$, $q_3$ be three consecutive (in clockwise orientation) vertices of  $NP(p)$ such that $q_1$ and $q_2$ are the endpoints of $S$ and let $\tilde q_i=T(q_i)$ for $i=1,2,3$. Then the points  
$\tilde q_1$, $\tilde  q_2$, $\tilde q_3$ are three consecutive  vertices of the Newton polygon of $\tilde p$ and the edge $\tilde S$ of $NP(\tilde p)$ with the endpoints $\tilde q_1$, $\tilde  q_2$ is the image of $S$ by $T$. 

By our choice of $\xi$ and $\eta$, we have 
$\langle\xi,q_1\rangle = \langle\xi,q_2\rangle > \langle\xi,q_3\rangle$ by convexity 
and $\langle\eta,q_2\rangle < \langle\eta,q_1\rangle$.
Replacing $\eta$ by $\eta-K\xi$, where $K$ is a sufficiently large integer, we can additionaly assume that $\langle\eta,q_2\rangle \leq \langle\eta,q_3\rangle$.
Then the edge $\tilde S$ is horizontal and $\tilde q_1$ and $\tilde q_3$ belong to the quadrant $\tilde q_2+\R_+^2$. 
Thus, by convexity the entire Newton polygon of $\tilde p$ is contained in this quadrant. 
See Figure \ref{sketch}. 
\begin{figure}[htpb]
\begin{center}
\begin{tikzpicture}
\draw[->,thick] (-1,0) -- (4,0);
\draw[->,thick] (0,-1) -- (0,3.5);
\coordinate (q3) at (2.8,0.4);
\coordinate (q2) at (2.5,2);
\coordinate (q1) at (1,3);
\draw (q3) -- (q2);
\draw[ultra thick, blue!70!black] (q2) -- (q1) node[pos=0.2, below left, yshift=4pt, black] {$S$};
\draw (q1) -- (0.3,2.5); 
\fill (q1) circle (1.5pt) node[above] {$q_1$};
\fill (q2) circle (1.5pt) node[right] {$q_2$};
\fill (q3) circle (1.5pt) node[above right] {$q_3$};
\draw[->,thick] (1.7,2.5) -- (2.5,3.6) node[pos=1, below, xshift=4pt, yshift=3pt , black] {$\xi$};
\draw[->,thick] (1.7,2.5) -- (0.5,1.8) node[pos=0.9, below, xshift=-4pt, black] {$\eta$};
\draw[->, bend left=30] (4.5,2) to node[midway, above] {$T$} (6.5,2);
\draw[->,thick] (7,0) -- (11,0);
\draw[->,thick] (8,-1) -- (8,3.5);
\coordinate (q3t) at (8.5,2.5);
\coordinate (q1t) at (10.5,-0.8);
\coordinate (q2t) at (7.5,-0.8);
\fill (q1t) circle (1.5pt) node[below] {$\tilde{q}_1$};
\fill (q2t) circle (1.5pt) node[below] {$\tilde{q}_2$};
\fill (q3t) circle (1.5pt) node[left] {$\tilde{q}_3$};
\draw (10.8,0.5) -- (q1t); 
\draw[ultra thick, blue!70!black] (q1t) -- (q2t) node[midway,below, black] {$\tilde{S}$};
\draw (q2t) -- (q3t); 
\fill[pattern=north east lines, pattern color=gray!60] 
(7.5,-0.8) -- (11,-0.8) -- (11,3.5) --(7.5,3.5)-- cycle;
\end{tikzpicture}
\caption{The Newton polygons of $p$ and of $\tilde{p}$.}\label{sketch}
\end{center}
\end{figure}
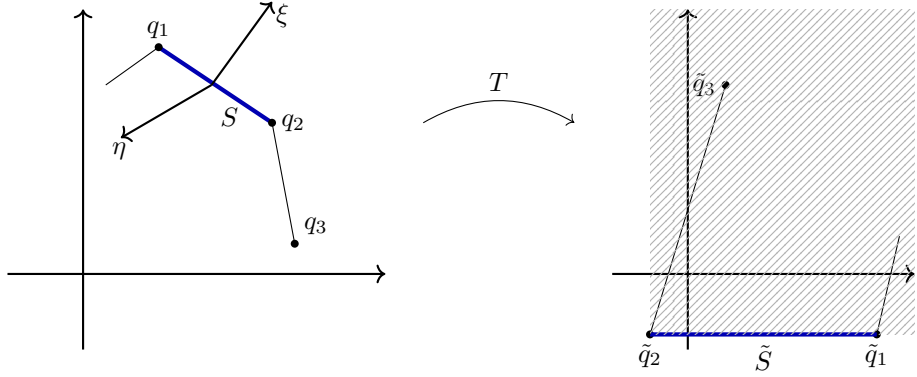

The edge $\tilde S$ is the Newton polygon of  $p_S\circ\sigma$. Moreover by~\eqref{monomial} and~\eqref{binomial} we get 
$$ 
(p_S\circ\sigma)(u,v)= u^{-A} v^{-B}  \prod_{i=1}^k (u - c_i)^{\nu_i}
$$ 
for suitable $A,B \in \Z$. 
By convexity argument the origin $(0,0) = T(0,0)$ belongs to $\tilde q_2+\R_+^2$, so $A\geq 0$. Since $(0,0)$ does not belong to the line containing $S$, we get that $(0,0)$ does not belong to the line containing $\tilde S$ and consequently $B>0$. 

Summing up 
\begin{equation}\label{local}
\tilde p(u,v)= u^{-A} v^{-B} \left[  \prod_{i=1}^k (u - c_i)^{\nu_i} + v \tilde p_1(u,v) \right]
\end{equation}
for some nonzero \emph{polynomial} $\tilde p_1\in \R[u,v]$, because the Newton polygon of the ``polynomial'' in square parenthesis is contained in the first quadrant and $\tilde S$ is horizontal. 
Since the edges $S$ and $\tilde S$ have the same number of lattice points, because $\sigma$ is an automorphism on $\Z^2$, we have 
\begin{equation}\label{67}
\overline N_S=\deg  \prod_{i=1}^k (u - c_i)^{\nu_i} = \sum_{i=1}^k \nu_i. 
\end{equation}

\medskip
Next we describe the effect of the transformation~\eqref{toric} to analytic parametrizations.  

Any analytic curve $\tilde\gamma:(\R,0)\ni t\to (u(t),v(t))\in\R^2$ such that 
$$
\begin{aligned}
u(t) & = c + \cdots \\ 
v(t) & = d t^m + \cdots
\end{aligned}\ , 
\quad  c,d\neq 0,\; m\in \Z,\; m \geq 1, 
$$
where dots mean terms of higher degrees, will be called a \emph{curve of type $(c,0)_m$}. 

Any meromorphic curve $\gamma:(\R,0)\ni t\to (x(t),y(t))\in\R^2$ such that 
\begin{equation*}
\begin{aligned}
x(t) & =  a t^{-m\xi_1}+ \cdots \\
y(t) & = b t^{-m\xi_2} + \cdots
\end{aligned}\ , \quad  a,b\neq 0,\;  m\in \Z,\; m\geq 1, 
\end{equation*}
where dots mean terms of higher degrees will be called a \emph{curve of type $(\xi, a^{-\xi_2}b^{\xi_1})$}. 

\medskip
\noindent
\textbf{Claim.}
If $\tilde\gamma$ is a curve of type $(c,0)_m$ then $\sigma\circ\tilde\gamma$ is a curve of type $(\xi, a^{-\xi_2} b^{\xi_1})_m$, with $a = c^{\eta_1}$ and $b = c^{\eta_2}$. 
Conversely, if $\gamma$ is a curve of type $(\xi, a^{-\xi_2} b^{\xi_1})_m$, then $\sigma^{-1}\circ\gamma$ is a curve of type $(c,0)_m$, with $c = a^{-\xi_2}b^{\xi_1}$. 
\begin{proof}
Let $\tilde\gamma(t)=(c+\cdots, d t^m+\cdots)$. 
Then 
$$
\begin{aligned}
\sigma(\tilde\gamma(t)) & =\left((c+\cdots)^{\eta_1} (d t^m+\cdots)^{-\xi_1}, (c+\cdots)^{\eta_2} (d t^m+\cdots)^{-\xi_2}\right)\\
& =\left(c^{\eta_1}d^{-\xi_1}t^{-m\xi_1}+\cdots, c^{\eta_2}d^{-\xi_2}t^{-m\xi_2}+\cdots\right). 
\end{aligned}
$$
Since $(c^{\eta_1}d^{-\xi_1})^{-\xi_2}(c^{\eta_2}d^{-\xi_2})^{\xi_1} = (c^{\eta_1})^{-\xi_2} (c^{\eta_2})^{\xi_1} = c$, it follows that $\sigma\circ\tilde\gamma$ is of type $(\xi, a^{-\xi_2} b^{x_1})_m$. 
Let $\gamma(t)=(a t^{-m\xi_1}+\cdots, b t^{-m\xi_2}+\cdots)$. 
Then $\sigma^{-1}(\gamma(t))$ is equal to 
$$
\begin{aligned}
\left((a t^{-m\xi_1}+\cdots)^{-\xi_2}(b t^{-m\xi_2}+\cdots)^{\xi_1}, (a t^{-m\xi_1}+\cdots)^{-\eta_2}(b t^{-m\xi_2}+\cdots)^{\eta_1}\right) \\
 = \big(a^{-\xi_2}b^{\xi_1}+\cdots, a^{-\eta_2}b^{\eta_1}t^{m(\xi_1\eta_2-\xi_2\eta_1)}+\cdots \big) \\ 
 = \big(a^{-\xi_2}b^{\xi_1}+\cdots, a^{-\eta_2}b^{\eta_1}t^m+\cdots\big),
\end{aligned}
$$ 
hence $\sigma^{-1}\circ\gamma(t)$ is of type $(a^{-\xi_2}b^{\xi_1},0)_m$. 
\end{proof}

It follows from the claim and~\eqref{local} that the number of branches at infinity 
of $p^{-1}(c)$ associated with the edge $S$ is equal to the number of analytic branches 
of the curve 
$$
\tilde P_c(u,v):=\prod_{i=1}^k (u - c_i)^{\nu_i} + v \tilde p_1(u,v)-cu^Av^B=0
$$ 
passing through points $(c_i,0)$ where $1\leq i\leq k$ and $c_i$ is a real number.  Hence it is enough to prove the following result. 

\begin{lemma}\label{local-bound}
Let $\tilde N_i(c)$ denotes the number of branches of the curve 
$\tilde P_c(u,v)=0$ passing through $(c_i,0)$, $c_i \in \R$. 
Then 
\begin{enumerate}[label={\textnormal{(\roman*)}}]
\item the function $\tilde N_i$  is bounded from above by $\nu_i$,
\item the function $\tilde N_i$ is constant $\mod{2}$ outside a finite set, 
\item if $\nu_i=1$ then $\tilde N_i=1$,
\item for every $c_0\in\R$ there exists $\epsilon >0$ such that the restrictions of $\tilde N_i$ to the intervals $(c_0-\epsilon, c_0)$ and $(c_0, c_0 + \epsilon)$ are constant.
If $\tilde N_i(c_0)>0$ then at least one of these restrictions is nonzero.
\end{enumerate}
\end{lemma}
We will prove this lemma later. 
First we complete the proof of Lemma~\ref{bounded}. 
We have 
\begin{equation}\label{local-global}
N_S=\sum_i \tilde N_i
\end{equation}
where the sum runs over $i\in\{1,\dots,k\}$ such that $c_i$ is real. 
Hence by~\eqref{67} and (i)~of Lemma~\ref{local-bound} we obtain~(i) of the lemma. 
Item~(ii) follows from~\eqref{local-global} and Statement (ii)~of Lemma~\ref{local-bound}.  
If $p$ is nondegenerate on $S$, then by (iii)~of Lemma~\ref{local-bound} $N_S$ is equal 
to the number of real roots of the polynomial $\prod_{i=1}^k (u - c_i)^{\nu_i}$ which in turn 
is congruent modulo 2 to the degree of this polynomial (see \eqref{67}). 
This gives (iii).  
If $S$ has no interior lattice points then  $k=1$ and $v_1=1$ in formula~\eqref{local}. 
Therefore by (ii)~of Lemma~\ref{local-bound} $N_S=\tilde N_1=1$.  Item (v) follows from~\eqref{local-global} and (iv)~of Lemma~\ref{local-bound}.
\end{proof}

\begin{proof}[Proof of Lemma~\ref{local-bound}]
The number of branches of the curve $\tilde P_c(u,v)=0$ passing through $(c_i,0)$ is bounded 
from above by the order of the polynomial $\tilde P_c(u,v)$ at $(c_i,0)$ which is less than or equal 
to $\nu_i$. 
This gives~(i). 
If $\nu_i=1$, then $(\tilde P_c)_u (c_i,0) \neq 0$, so by the implicit function theorem, there is exactly one smooth branch of $\tilde P_c(u,v) = 0$ passing through $(c_i,0)$, that can be written as a graphic $(u(t), t)$. 
In particular, this branch is of type $(c_i,0)_1$. 
This gives~(iii).  
	
In order to prove (ii) and (iv) we will use resolution of singularities. Consider a rational function 
$ F:\R^2\to \R\mathbb{P}^1$,  $F(u,v)=u^{-A}v^{-B}\bigl(\prod_{i=1}^k (u - c_i)^{\nu_i} + v \tilde p_1(u,v)\bigr)$. As the function $F$ is not defined at $(c_i,0)$ it may not extend continuously 
at this point.  

However one can find a smooth algebraic manifold $M$ and a proper mapping $\pi:M\to \R^2$ 
which is a composition of blowing ups such that $\pi$ restricted to $M\setminus \pi^{-1}(c_i,0)$
is a biregular map onto $\R^2\setminus \{(c_i,0)\}$ and the set $E=\pi^{-1}(c_i,0)$ (called an exceptional divisor) is a connected union of irreducible algebraic components $E_1\cup \dots \cup E_s$, each biregular with the projective line $\R\mathbb{P}^1$.  
Moreover there exists a neighbourhood $U$ of $(c_i,0)$ and a regular function $\hat F: \pi^{-1}(U)\to \R\mathbb{P}^1$ such that $\hat F=F\circ\pi$ on the set $ \pi^{-1}(U)\setminus E$. This proces is called a resolving of indeterminacy of $F$ at $(c_i,0)$.  For a reference on this technique, see, for instance, \cite[Corollary 1.76]{kollar}.

We will also use the property that for any analytic branch $\gamma$ passing through $(c_i,0)$ 
there exists a unique point $P$ of $E$ and an analytic branch $\hat\gamma$ passing through 
$P$ such that $\gamma =\pi(\hat\gamma)$. We will call $\hat\gamma$ the lift of $\gamma$.

For any $j\in\{1,\dots,s\}$ the restriction $\hat F|_{E_j}$ can be identified 
with a rational function $H:\R\mathbb{P}^1\to\R\mathbb{P}^1$.  
	
\medskip\noindent
\textbf{Claim.}  The function $c\to \# H^{-1}(c)$ is constant $\mod{2}$ outside a finite set. 
	
\begin{proof}[Proof of the claim]
If $H$ is constant then there is nothing to prove.  Thus assume that $H$ is nonconstant. 
Let $c_0$ be any point of $\R\mathbb{P}^1$. We will show that the function $c\to \# H^{-1}(c)$
preserves parity while passing through $c_0$.  Applying a homography at the source and at the 
target we may assume that $c_0=0$ and the preimage $H^{-1}(c_0)=:\{z_1,\dots,z_m\}$ does 
not contain the point at infinity of $\R\mathbb{P}^1$.  For every point $z_l$ for $1\leq l\leq m$ 
there exists a neighborhood $U_l$ such that the function $H$ is monotonous on the left of $z_l$
and on the right of $z_l$. 
Let $\epsilon >0$ be small enough. 
We have three possibilities: 
\begin{itemize} 
\item[-] if $z_l$ is a local minimum of $H$ then $H^{-1}(c)\cap U_l$ is empty for $c\in (-\epsilon,0)$ and 
has two points for $c\in (0,\epsilon)$, 
\item[-] if $z_l$ is a local maximum of $H$ then $H^{-1}(c)\cap U_l$ has two points  for $c\in (-\epsilon,0)$ and is empty for $c\in (0,\epsilon)$, 
\item[-] if $H$ is increasing or decreasing in $U_l$ then  $\#(H^{-1}(c)\cap U_l)=1 $ for $c\in (-\epsilon,\epsilon)$. 
\end{itemize}
		
The function $H$ is separated from $0$ in the complement of $U_1\cup\dots \cup U_m$. Thus shrinking $\epsilon$ if necessary we have: $\#H^{-1}(c)$ is constant for $c\in (-\epsilon,0)$,   $\#H^{-1}(c)$ is constant for $c\in (0,\epsilon)$,	and $\#H^{-1}(c)$ preserves parity for $c\in (-\epsilon,0)\cup (0,\epsilon)$.  This proves the claim.
\end{proof}
	
	
It follows from the claim that for any $c_0\in\R$ there exists $\epsilon>0$ such that the function 
$c\to \# (\hat F^{-1}(c)\cap E)$ is constant $\mod 2$ for $c\in (-\epsilon,0)\cup (0,\epsilon)$.
We may also assume, shrinking $\epsilon>0$ if necessary, that every $c\in (-\epsilon,0)\cup (0,\epsilon)$ is a regular value of $\hat F$ restricted to any component $E_j$ of the exceptional divisor $E$. This implies that the curve $\hat F^{-1}(c)$ intersects $E$ transversally, so the number 
of branches of the curve $\hat F^{-1}(c)$ passing through $E$ is equal to $\# (\hat F^{-1}(c)\cap E)$.
Since the branches of the curve $\tilde P_c(u,v)=0$ passing through $(c_i,0)$ lift to the branches 
of $\hat F^{-1}(c)$ passing through $E$, we get (ii). 
	
In order to prove (iv) assume that $\tilde N_i(c_0)>0$. 
Then there is at least one branch of the curve $\tilde P_{c_0}(u,v)=0$ passing through $(c_i,0)$. Since the function $\hat F$ has value $c_0$ on the lift of this branch,  it has value $c_0$ at some point of $E$. 
Observe also that $\hat F(u,0)=\infty$ if $u\notin \{c_1,\dots, c_k\}$. 
The lift of a branch $v=0$ to $M$ intersects $E$, hence by continuity of $\hat F$ this function 
attains value $\infty$ at some point of $E$.  The image of every $E_j$ by $\hat F$ is either a point 
or a connected infinite subset of $\R\mathbb{P}^1$.  Since $E$ is a connected set, $\hat F(E)$ is connected too. 
Moreover $\hat F(E)$ contains at least two points: $c_0$ and $\infty$. 
Thus there exists $E_j$ such that $\hat F$ restricted to $E_j$ is nonconstant and $c_0\in \hat F(E_j)$. 
Hence at least one of the functions $H$ considered in the claim is nonconstant and thus is nonzero in the interval $(c_0,c_0+\epsilon)$ or in the interval $(c_0-\epsilon,c_0)$ (see the proof of the claim). This gives (iv).
\end{proof}

\begin{corollary}\label{type}
Let $p(x,y)$ be a convenient polynomial such that $p$ is nondegenerate on the edge $S$ of $NP(p)$. 
Then $m$ of \emph{Definition~\ref{mab}} can be taken to be $1$. 
\end{corollary}
\begin{proof} 
The positive integer $m$ of Definition~\ref{mab} is the same $m$ appearing in the definition of the curves of type $(c_i,0)_m$ and $(\xi, A^{-\xi_2} B^{\xi_1})_m$ from the proof of Lemma~\ref{bounded}. 
In the nondegenerate case, all local analytic curves of type $(c_i,0)_m$ have $m = 1$, according to the proof of Lemma \ref{local-bound}. 
Hence for all the branches of infinity of $p(x,y) = 0$ associated with $S$, we will have $m = 1$ in Definition~\ref{mab}. 
\end{proof}

\subsection{The type of $N_S$}
Let $C_1, \ldots, C_k$ be semialgebraic subsets of $\R$ and $m_1, \ldots, m_k$ be real numbers. 
As usually we define a \emph{constructible function} $f: \R \to \R$ as
\begin{equation}\label{simples}
f(x) = \sum_{i = 1}^k m_i \ind_{C_i}(x)
\end{equation}
By considering Thom's result on the finiteness of the bifurcation values of a polynomial $p(x,y)$, it follows in particular that for a given outer edge of $NP(p)$ convenient, the function $N_S$ has finitely many number of jumps, and so it is a constructible function with $k = 2 l + 1$ and 
$$
C_1 = (-\infty, c_1), \ C_2 = \{c_1\}, \ C_3 =  (c_1, c_2),  \ldots, \ C_{k-1} = \{c_l\}, \ C_k = (c_{l}, \infty), 
$$
for $c_1 < c_2 < \cdots < c_{l}$ suitable real numbers. 
In this case we say that \emph{$N_S$ is of type $m_1 {\bf m_2}m_3\cdots m_{k-2} {\bf m_{k-1}} m_k$} 
(here we are not excluding the possibility of $m_{i-1} = m_i = m_{i+1}$ for some $i$, and so the ``type" is not unique). 
When $N_S$ is \emph{constant}, we simply say so. 
Clearly function $N$ has exactly the same property than each function $N_S$. 
And we can define the type of $N$ similarly. 
When $p$ is a submersion, see again Remark \ref{45r}, and we use the set $\{c_1, \ldots, c_l\}$ of the definition of $N$ to agree with $B(p)$, it follows that the type $m_1 {\bf m_2}m_3\cdots m_{k-2} {\bf m_{k-1}} m_k$, $k = 2 l + 1$, of $N$ agrees with the definition of \emph{type of $p$} given in \cite{BFOO}. 

\subsection{The Euler integral}
Let $f: \R \to \R$ be as in~\eqref{simples} with $C_1, \ldots, C_k$ connected. 
The \emph{Euler integral} of $f$ is defined by 
$$
\int f d \chi = \sum_{i = 1}^k m_i \chi(C_i), 
$$
where $\chi(A)$ stands for the Euler characteristic of the connected subset $A$ of $\R$. 
So for $N_S$, for instance, $\chi(C_{2i}) = 1$ and $\chi(C_{2 i + 1}) = -1$, $i = 0, \ldots, l$. 

The following theorem establishes the relationship between the \emph{degree of the gradient vector field of $p$ at infinity}, denoted by $\deg_{\infty} \nabla p$ and the Euler integral. 
Recall that for a given polynomial function $p:\R^2\to \R$ having a compact set of critical points, the \emph{degree of the gradient vector field of $p$ at infinity}, denoted by $\deg_{\infty} \nabla p$, is defined by the topological degree of the mapping 
	$ S_r\ni a \to \frac{\nabla  p(a)}{\| \nabla p(a)\|}\in S_1$ where $S_r$ is a circle with the center at the origin and radius $r$ big enough.
Equivalently, $\deg_{\infty} \nabla p$ agrees with the sum of the index of $\nabla p$ over all its finite equilibrium points. 

The following result is a formulation of S\c{e}kalski's theorem \cite{sekalski}.  

\begin{theorem}[{\cite[Theorem 7]{Gii}}]
Let $p(x,y)$ be a polynomial with finite set of critical points. 
Then 
$$ 
\deg_{\infty} \nabla p = 1 + \int N d\chi.
$$
\end{theorem}

In the particular case of a submersion, the degree in the left is zero. 
Hence, we easily conclude the following result.

\begin{corollary}\label{corSek}
If $p(x,y)$ is a polynomial submersion, then 
$$
\int N d \chi = -1.
$$
\end{corollary}

\begin{remark}\label{rem3}
As a direct consequence, when $p$ is a polynomial submersion and $N$ is constant, it follows that $N=1$, and so $p$ is a locally trivial fibration. 
From Lemma \ref{bounded}, this is the case when $p$ a polynomial submersion and $p$ is nondegenerate. 
\end{remark}

\begin{example}
Let $p(x,y) = 1+ x + y + x^2 y + 2 x y^2 + y^3$. 
The Newton polygon of $p$ is depicted in Fig~\ref{broughton}. 
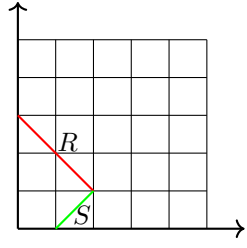
\begin{figure}[htbp]
\begin{center}
{
\begin{tikzpicture}
\draw[-] (0.5,0)--(0.5,2.5);
\draw[-] (1,0)--(1,2.5);
\draw[-] (1.5,0)--(1.5,2.5);
\draw[-] (2,0)--(2,2.5);
\draw[-] (2.5,0)--(2.5,2.5);
\draw[-] (0,0.5)--(2.5,0.5);
\draw[-] (0,1)--(2.5,1);
\draw[-] (0,1.5)--(2.5,1.5);
\draw[-] (0,2)--(2.5,2);
\draw[-] (0,2.5)--(2.5,2.5);
\draw[->,thick](0,0)--(0,3);
\draw[->,thick](0,0)--(3,0);
\draw[-,red,thick](0,1.5)--(1,0.5);
\draw[-,green,thick](1,0.5)--(0.5,0);
\node at (0.66,1.15){{$R$}};
\node at (0.85,0.18){{$S$}};
\end{tikzpicture}}
\end{center}
\caption{The Newton polygon of $p(x,y)$.}\label{broughton}
\end{figure}
It has two outer edges $R$ and $S$ as in the picture. 

The polynomial $p$ is nondegenerate on the edge $S$ and $\overline{N}_S = 1$, so $N_S = \ind_{\R}$ from Lemma~\ref{bounded}. 
Edge $R$ is such that $\overline{N}_R = 2$, so by the same lemma, $N_R(c) \in \{0,1,2\}$ for any $c \in \R$. 

Actually, we can write $p(x,y) = 1 + q(x+y,y)$ with $q(x,y) = x (1 + x y)$. 
The polynomial $q(x,y)$, largely known as Broughton example, is a submersion with $B(q) = \{0\}$, and such that $q^{-1}(0)$ has $3$ connected components and $q^{-1}(c)$ has two connected components for any $c\neq 0$. 
Then $N_R = \ind_{(-\infty, 1)} + 2 \ind_{\{1\}} + \ind_{(1, \infty)}$ and $N = 2 \ind_{(-\infty, 1)} + 3\ind_{\{1\}} + 2 \ind_{(1, \infty)}$
The Euler integral of $N$ calculates as $2 \cdot (-1) + 3 \cdot (1) + 2 \cdot (-1) = -1$. 
\end{example}

The following immediate consequence of Corollary~\ref{corSek} recovers Theorem 2.12 in \cite{BFOO}, where it was proved without using the Euler integral. 

\begin{corollary}
If $p$ is a polynomial submersion of type $m_1 {\bf m_2}m_3\cdots {\bf m_{k-1}} m_k$, $k = 2 l + 1$, then 
$$
\sum_{i=0}^{\frac{k-1}{2}} m_{2 i + 1} = 1 + \sum_{i=1}^{\frac{k-1}{2}} m_{2 i}. 
$$ 
\end{corollary}

\section{Global dynamics of Hamiltonian vector field for eliminating Jacobian mates}\label{ds}
\subsection{Poincar\'e compactification}\label{Pc}
Any planar polynomial vector field $\mathcal{X}$ can be compactified by means of the \emph{Poincar\'e compactification of $\mathcal{X}$}, giving rise to a vector field $p(\mathcal{X})$ in $\R \mathbb P^2$, in such a way that the invariant flow on the line at infinity corresponds to the behavior of $\mathcal{X}$ at infinity, and the flow on the interior of the disc is the transformed flow up to multiplication by a suitable positive function. 

There is an appealing geometric construction of this compactification, summarized as follows. 
We first identify $\R^2$ with the affine tangent plane $\pi$ of the bi-dimensional unit sphere $\mathbb{S}_2$ in the north pole. 
Each straight line non parallel to $\pi$ through the origin of $\R^3$ cuts $\pi$ in one point and $\mathbb{S}_2$ in two points. 
This gives two analytic diffeomorphisms $f^+$ and $f^-$ from $\pi$ onto the north and south hemispheres of $\mathbb{S}_2$, respectively. 
Through $f^+$ and $f^-$ we get two copies of $\mathcal{X}$ in $\mathbb{S}_2$, one in the north and one in the south hemisphere, respectively. 
It turns out that by multiplying these copies by $x_3^{m-1}$, where $m$ is the degree of $\mathcal{X}$ and $(x_1, x_2, x_3)$ are the coordinates of $\R^3$, we get a polynomial vector field in the entire $\mathbb{S}_2$, called the \emph{Poincar\'e compactification of $\mathcal{X}$}, whose restrictions to each of the hemisphere is topologically equivalent to $\mathcal{X}$. 
By construction, the equator of $\mathbb{S}_2$ is invariant by the flow. 
The points at the equator are said the \emph{infinite} points of $\R^2$, and the singular points of the Poincar\'e compactification of $\mathcal{X}$ at the equator are said the \emph{infinite singular points of $\mathcal{X}$}. 
It is clear that $\mathcal{X}$ in $\R^2$ together with its infinite points is fully qualitatively understood in the unit disc, our model for $\R \mathbb P^2$: simply project the closure of the north hemisphere of $\mathbb{S}_2$ onto the disc. 
Since the antipodal points of $\mathbb{S}_2$ correspond to the same point in the original plane $\R^2$, in order to study the Poincar\'e compactification of $\mathcal{X}$ we only need the local canonical charts $(U_i,\phi_i)$, where $U_i = \{(x_1,x_2,x_3)\in \mathbb{S}_2\ |\ x_i > 0\}$, with $\phi_i(x) = (x_{j_1}/x_i, x_{j_2}/x_i)$, $j_1 < j_2 \in \{1,2,3\}\setminus \{i\}$, $i=1,2,3$. 
It is simple to see that both composite applications $\phi_i \circ f^+(x_1,x_2)$ and $\phi_i \circ f^-(x_1,x_2)$ give the expression $(x_j/x_i,1/x_i)$, $j\neq i$, $i, j \in \{1,2\}$.  
So if in the original plane we want to study the infinite point given by a direction $w = (w_1, w_2)$, with $w_i \neq 0$, we use the transformation $(u,v) = (x_j/x_i,1/x_i)$, $j\neq i$, defined in $x_i\neq 0$, nothing more than the \emph{homogeneous coordinates} of $\R \mathbb P^2$. 
For more details on the Poincar\'e compactification, we address the reader to \cite[Chapter 5]{DLA}. 

\subsection{The Hamiltonian vector field, hyperbolic sectors at infinity, inseparable orbits and Jacobian mates}
For any polynomial function $f: \R^2 \to \R$, we define the \emph{Hamiltonian vector field associated with $f$} by $\nabla f^{\perp} = \left(-f_y, f_x\right)$. 
The main property of $\nabla f^{\perp}$ is that each of its orbits is contained in a fiber of $f$, and when $f$ is a submersion, each orbit agrees with a connected component of a fiber, as can be readily checked. 

We will freely use the \emph{sectorial decomposition} of isolated singular points of polynomial vector fields into \emph{hyperbolic}, \emph{elliptic}, and \emph{parabolic sectors}; see \cite[Chapter 1]{DLA} for details. 

For a given hyperbolic sector $h$ of an isolated singular point of a vector field, the two arcs of orbits in the boundary of $h$ are said the \emph{separatrices} of $h$. 
We denote them by $s_1(h)$ and $s_2(h)$. 
Next result is Theorem 3.5 of \cite{BFOO}. 

\begin{theorem}[\cite{BFOO}]\label{hyperbolic}
Let $p(x,y)$ be a polynomial submersion such that $\gamma_1$,..., $\gamma_{k+1}$, $k > 1$, are distinct orbits of $\nabla p^{\perp}$, contained in a fiber of $p$ such that, for $i = 1, \ldots, k$, 
\begin{enumerate}[label={\textnormal{(\roman*)}}]
\item\label{primer} there exists a hyperbolic sector $h_i$ of $\nabla p^{\perp}$ at the infinite singular point $z_i$ such that $\gamma_i$ contains $s_1(h_i)$ and $\gamma_{i+1}$ contains $s_2(h_i)$.  
\item\label{segon} $z_i$ is the $\omega$-limit of $\gamma_i$ and the $\alpha$-limit of $\gamma_{i+1}$.  
\end{enumerate}
If $z = z_1 = z_{k}$ and there is a branch of an algebraic curve at the infinite point $z$ with one half-branch contained in $h_1$ and the other contained in $h_{k}$, then $p$ has no Jacobian mates. 
\end{theorem}
See Figure~\ref{hyperbolics} 
\begin{figure}[htbp]
\begin{center}
\psfrag{g1}{\tiny$\gamma_1$}
\psfrag{g2}{\tiny$\gamma_2$}
\psfrag{g3}{\tiny$\gamma_3$}
\psfrag{g4}{\tiny$\gamma_4$}
\psfrag{g5}{\tiny$\gamma_5$}
\psfrag{g6}{\tiny$\gamma_6$}
\psfrag{h1}{\tiny$h_1$}
\psfrag{h2}{\tiny$h_2$}
\psfrag{h3}{\tiny$h_3$}
\psfrag{h4}{\tiny$h_4$}
\psfrag{h5}{\tiny$h_5$}
\psfrag{z1}{\tiny$z_1$}
\psfrag{z2=z3}{\hspace{-.2cm}\raisebox{.2cm} {\tiny$z_2 = z_3$}}
\psfrag{z4}{\tiny$z_4$}
\psfrag{z5}{\tiny$z_5$}
\includegraphics[scale=.28]{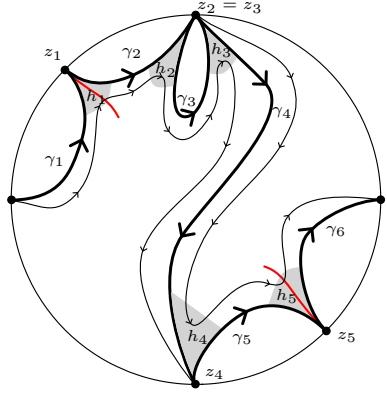}
\end{center}
\caption{A chain of hyperbolic sectors $h_1, \ldots, h_5$ representing conditions~\ref{primer} and~\ref{segon} of theorem \ref{hyperbolic}. 
In red the half-branches of an algebraic curve entering both $h_1$ and $h_5$, representing the last assumption of that theorem}\label{hyperbolics}
\end{figure}
for an example of the configuration described in the theorem. 

Theorem \ref{hyperbolic} will be one of the essential tools we shall use in Section~\ref{mr}, but it is quite difficult to check its hypotheses for a given polynomial submersion.
We therefore develop some sufficient conditions that allow us to do this more directly by using information from the Newton polygon of the polynomial.
For this, one task is to guarantee the existence of hyperbolic sectors at infinity with a branch of an algebraic curve entering them as in the theorem. 
The other is to guarantee that these hyperbolic sectors ``connect'' to each other through a chain of hyperbolic sectors as in \ref{primer} and \ref{segon} of the theorem. 
We close this section with Lemma \ref{connects}, which helps address the second issue.
In the next section, we deal with the first one and combine the resulting statements.

But before stating the lemma, we shall need one further ingredient. 
The orbits $\gamma_i$ and $\gamma_{i+1}$ in Theorem \ref{hyperbolic} are ``inseparable'' in the following sense. 
Assume that $p$ is a submersion. 
Two orbits of $\nabla p^{\perp}$, or equivalently, two connected components of a fiber of $p$, $\alpha$ and $\beta$ are said \emph{inseparable} when for any transversal sections $S$ and $T$ through $\alpha$ and $\beta$, respectively, there exists a third orbit cutting both $S$ and $T$. 
Clearly orbits containing the separatrices $s_1(h)$ and $s_2(h)$ of a hyperbolic sector $h$ are inseparable. 
It is also not difficult to conclude, from the definitions of inseparability and hyperbolic sector that, conversely, given two inseparable orbits $\alpha$ and $\beta$, there exists a sequence of orbits $\gamma_1, \ldots, \gamma_k$ as in~\ref{primer} and~\ref{segon} of Theorem \ref{hyperbolic} such that $\gamma_1 = \alpha$ and $\gamma_k = \beta$ or $\gamma_1 = \beta$ and $\gamma_k = \alpha$. 
For instance, in Figure \ref{hyperbolics} the collections $\{\gamma_i\}_{i\in \{1,2,3,4\}}$ and $\{\gamma_j\}_{j\in \{4,5,6\}}$ are formed by pairwise inseparable orbits. 
For an overview on these remarks for general polynomial vector fields, the reader can consult \cite{JL}. 

\begin{lemma}\label{connects}
Let $p(x,y)$ be a polynomial submersion and $c_0$ be a point of discontinuity of the function $N$ such that either $N(c_0) = 3$ or there exists $\epsilon >  0$ such that $N(c) = 1$ for all $c$ in $(c_0-\epsilon, c_0)$ or in $(c_0, c_0 + \epsilon)$, with $N(c_0)>2$. 
Let $h_{a}$ and $h_{b}$ be two given hyperbolic sectors of $\nabla p^{\perp}$ at infinity whose separatrices are contained in $p^{-1}(c_0)$. 
Assume that $h_{a}$ and $h_{b}$ do not share separatrices. 
Then $h_{a}$ and $h_{b}$ connect in the sense of~\ref{primer} and~\ref{segon} of Theorem~\ref{hyperbolic}, that is, one is $h_1$ and the other is $h_k$ in the notation of the theorem. 
\end{lemma}
\begin{proof}
Without loss of generality, we may assume that $c_0=0$. 
	
First we study the case $N(0) > 2$ and $N(c) = 1$ for all $c \in (0, \epsilon)$. 
Clearly, the case $N(c) =1$ for all $c \in (-\epsilon, 0)$ is analogous. 
We will prove that the connected components of $p^{-1}(c_0)$ are pairwise inseparable, which is enough to conclude the lemma in this case because the inseparability will provide the desired connection between the hyperbolic sectors at infinity, as above commented. 

Indeed, let $\alpha$ and $\beta$ be  two connected components of $p^{-1}(c_0)$, and let $S$ and $T$ be transversal sections through $\alpha$ and $\beta$, respectively. 
For any $0 < \bar{\epsilon} \leq \epsilon$ small enough it follows that the intersection $p(S)\cap p(T)\cap (0,\bar{\epsilon})$ is not empty. 
Then, for each $c$ in this intersection, the connected fiber $p^{-1}(c)$ must cut both $S$ and $T$. 
That is, $\alpha$ and $\beta$ are inseparable. 

As for the other case, that is, when $N(0) = 3$, it follows that one separatrix of $h_a$ and one separatrix of $h_b$ must be contained in the same connected component $\alpha$ of $p^{-1}(0)$, otherwise $N(0) > 3$. 
Without loss of generality, we can assume that the curve $\alpha$ as an orbit of $\nabla p^{\perp}$ is oriented such that $\alpha(t)$ moves towards $h_{b}$ as $t$ increases. 
Then, in the notation of Theorem~\ref{hyperbolic}, we put $\gamma_2 = \alpha$, and $\gamma_1$ and $\gamma_3$ to be the orbits containing the remaining separatrices of $h_{a}$ and $h_{b}$, respectively. 
\end{proof}

\section{Algebraic Geometry meets dynamical systems}\label{agmds}
\subsection{Special edges}
\begin{theorem}\label{lattice_points}
Let $p(x,y)$ be a polynomial submersion with a convenient Newton polygon that has an outer edge $S$ with exactly one interior lattice point. Let $\xi=(\xi_1,\xi_2)$ be the outer normal vector to $S$ and let $B=\langle \xi, q \rangle$ where $q=(i_0,j_0)$ is any endpoint of $S$. 
Assume that $N_S$ is nonconstant. 
If the coefficient of $x^{i_0}y^{j_0}$ is positive then: 
\begin{enumerate}[label={\textnormal{(\roman*)}}]
\item for $B$ odd $N_S$ is of type $1{\bf b}1$ where $0\leq  b \leq 2$,  
\item for $B$ even $N_S$ is of type $0{\bf b}2$ where $0\leq b \leq 2$.
\end{enumerate}
Let $c_0$ be the point of discontinuity of $N_S$. 
If $b=1$  then the two half--branches at infinity  of $p(x,y)=c_0$ associated with $S$ bound a germ of an open set $U$ such that $p(a)\to c_0$ for $\| a\|\to \infty$, $a\in U$. 
If $b=2$ then two half--branches at infinity of $p(x,y)=c_0$ associated with $S$ bound a germ of an open set $U_1$ and the remainnig two half--branches at infinity of $p(x,y)=c_0$ associated with $S$ bound a germ of an open set $U_2$. 
We also have  $p(a)\to c_0$ for $\| a\|\to \infty$,  $a\in U_1\cup U_2$. 
In particular, the sets $U_1$ and $U_2$ are hyperbolic sectors of $\nabla p^{\perp}$ at the same point at infinity not sharing separatrices. 
If $B$ is odd then $p(a)<c_0$ for $a\in U_1$ and   $p(a)>c_0$ for $a\in U_2$. 
If $B$ is even  then $p(a)<c_0$ for $a\in U_1\cup U_2$.
Moreover, there exists a branch of an algebraic curve $\gamma$ at infinity having one half-branch contained in $U_1$ and the other contained in $U_2$. 
\end{theorem}
\begin{proof}
Without loss of generality we can assume that $c_0=0$ by replacing $p$ with $p-c_0$ if necessary. 
If $p$ is nondegenerate on $S$ then by (iii) of Lemma~\ref{bounded} $N_S$ is constant.  Therefore $p$ is degenerate on $S$. Given the assumpion on the number of interior lattice point on $S$, $p_S(x,y)$ has the form $\delta x^r y^s (y^{\xi_1} - c_1 x^{\xi_2})^2$. 
In particular, it readily follows that $B = \langle \xi, (r, s)\rangle + 2 \xi_1 \xi_2$. 
Without loss of generality we may assume that $\delta = 1$. 
	
Using the substitution $\sigma$ from the proof of Lemma~\ref{bounded} we get $p(\sigma(u,v)) = u^{-A}v^{-B}\bigl((u-c_1)^2 +  v \hat p_1(u,v)\bigr)$, according to \eqref{local}, for suitable $A \geq 0$ and $B = \langle \xi, (r,s)\rangle + 2 \xi_1 \xi_2$, from \eqref{monomial} and \eqref{binomial}, agreeing with the $B$ from the hypothesis. 
As in the proof of that lemma the problem of counting the branches at infinity of a curve $p(x,y)=c$ that are associated with the edge $S$ reduces to counting the branches at $(c_1,0)$ of a meromorphic curve $p(\sigma(u,v))=c$. 
Substituting $u$ by $u+c_1$ we get $p(\sigma(u+c_1,v))=(u+c_1)^{-A}v^{-B}\bigl(u^2+v\hat p_1(u+c_1,v)\bigr)=v^{-B}H(u,v)$ where $H(u,v)$ is an analytic function defined  in the neighborhood of the origin such that $H(u,0)=u^2+\mbox{higher order terms}$.
	
By the Splitting Lemma (see \cite[Lemma C.6.1]{H}), after an analytic changing of coordinates of the form $\phi(u,v)=(\phi_1(u,v),v)$, we have $H(\phi_1(u,v),v)=u^2+h(v)$ where $h$ is a germ at the origin of an analytic function, $h(0)=0$. 
	
Let $\Psi(u,v)=\sigma(\phi_1(u,v)+c_1,v)$. 
Then  
$$ 
f(u,v):=p(\Psi(u,v))=v^{-B}(u^2+h(v)) 
$$
for $(u,v)$ close to the origin such that $v\neq0$. 
	
The function $h$ does not vanish identically, since otherwise $\Psi(0,v)$ would be a double branch at infinity of the curve $p(x,y)=0$, a contradiction because $p$ is a submersion. 
Thus $h(v) = av^m + \mbox{higher order monomials}$, $a\neq 0$. 
	
The rest of the proof relies on classification of germ of analytic curves of the form $F(u,v)=u^2+r(v)=0$ where $r(v)=ev^s+\mbox{higher order monomials}$, $e\neq 0$. 
They are of three types: 
\begin{description}
\item[$A_{{\rm even}}^+$] if $s$ is even and $e>0$ then $F(u,v)=0$ has no real branches, 
\item[$A_{\rm even}^-$] if $s$ is even and $e<0$ then $F(u,v) = 0$ has two real branches $u=-\sqrt{-r(v)}$ and $u=\sqrt{-r(v)}$. 
In particular, $F(u,v) = 0$ has two half-branches in the plane $v>0$ and two half-branches in the plane $v<0$.
\item[$A_{\rm odd}$] if $s$ is odd then $F(u,v)=0$ has one real branch. 
It has two half branches $u=-\sqrt{-r(v)}$ and $u=\sqrt{-r(v)}$ in the half-plane $v>0$ (respectively, in the half plane $v<0$) if $e$ is negative (respectively, $e$ is positive). 
\end{description}
	
We now apply this remark to the meromorphic curve $f(u,v)=c$ or equivalently to the analytic curve $u^2+h(v)-cv^B=0$.  
If $B\geq m$ then the term $-cv^B$ does not alter the type of the singularity for $|c|<|a|$. 
In this case the function $N_S$ is locally constant at zero and is therefore continuous at $c_0=0$. Therefore $B<m$ and we consider the branches at the origin of the curve $F_c(u,v)=0$ where $F_c(u,v)=u^2-cv^{B}+av^m+\mbox{higher order monomials in } v$.
	
Assume that $B$ is odd. 
In this case a singularity $F_c(u,v)=0$ is of type $A_{\rm odd}$ for $c\neq 0$ which implies that $N_S$ is of type $1{\bf b}1$. 
Now assume that $B$ is even. In this case a singularity $F_c(u,v)=0$ is of type $A_{\rm even}^+$ for $c<0$ and of type $A_{\rm even}^-$ for $c>0$ which implies that $N_S$ is of type $0{\bf b}2$. 
This concludes (i) and (ii) of the theorem.

\medskip

\noindent
\textbf{Claim~1.} Assume that the curve $F_0(u,v)=0$ has two half-branches in the upper half-plane. 
Consider the set 
$$
U_{\epsilon}^+=\{\,(u,v)\in \R^2:  -\sqrt{-h(v)} < u < \sqrt{-h(v)},\ 0< v<\epsilon\,\}
$$ 
where $\epsilon>0$ is small enough. 
Then the function $f$ has a constant sign in $U_{\epsilon}^+$ and $f(u,v)\to 0$ as $(u,v)\to(0,0)$, $(u,v)\in U_{\epsilon}^+$. 
\begin{proof}[Proof of Claim 1]
For any point $(u,v)\in U_{\epsilon}^+$ we have $u^2 < -h(v)$. 
Thus $f(u,v)  < 0$.
Since $f(0,v)=av^{m-B}+\mbox{higher order terms}$, we have  $\lim_{v\to 0}f(0,v)=0$, 
Thus by inequalities $f(0,v)\leq f(u,v)<0$ we get that $f(u,v)\to 0$ as $(u,v)\to(0,0)$, $(u,v)\in U_{\epsilon}^+$. 
\end{proof} 
	
\medskip
	
An obvious modification of the above proof gives an analogous statement for the lower half-plane. 
	
\medskip

\noindent
\textbf{Claim~2.} 
Assume that the curve $F_0(u,v)=0$ has two half-branches in the lower half-plane. Consider the set $$
U_{\epsilon}^-=\{\,(u,v)\in \R^2:  -\sqrt{-h(v)} < u < \sqrt{-h(v)},\ -\epsilon<v<0\,\}
$$ 
where $\epsilon>0$ is small enough. Then the function $f$ has a constant sign in $U_{\epsilon}^-$ and $f(u,v)\to 0$ as $(u,v)\to(0,0)$, $(u,v)\in U_{\epsilon}^-$. 
	
\medskip
	
If $b=1$ then the curve germ $F_0(u,v)=0$ is of type $A_{\rm odd}$.  
This curve germ has two half-branches in the upper half-plane or two half-branches in the lower half-plane. 
Then $\Psi(U_{\epsilon}^+)$ or $\Psi(U_{\epsilon}^-)$ is a germ of an open set in  $\R^2$ bounded by half-branches at infinity of $p(x,y)=0$ associated with $S$.  	
	
\medskip
	
If $b=2$ then the curve germ $F_0(u,v)=0$ is of type $A_{\rm even}^-$.  In particular $m$ is even and $a<0$. 
This curve germ has two half-branches in the upper half-plane and two half-branches in the lower half-plane. 
Then $U_1 = \Psi(U_{\epsilon}^+)$ and $U_2 =\Psi(U_{\epsilon}^-)$ are germs of open sets in $\R^2$ bounded by half-branches at infinity of $p(x,y)=0$ associated with $S$. 
The image of $u=0$ by  $\Psi$ is a branch at infinity approaching both $U_1$ and $U_2$. 
Finally if $B$ is odd then the function $f(0,v)=a v^{m-B}+\cdots$ changes sign at $0$ from $+$ to $-$ and if $B$ is even then it is negative for $v\neq0$ close to 0. 
This proves the last statement of the theorem. 
\end{proof}

\begin{remark}\label{localcharts}
It follows from the proof of the theorem that in the homogeneous local charts (the image of) $U_1$ is contained in the upper half plane and (the image of) $U_2$ is contained in the lower half-plane. 
This fact is important when we use the theorem to study the global foliation given by a submersion $p$. 
As an example, see our reasonings to prove the non-existence of Jacobian mates for $p$ when its Newton polygon is as in case (c) of Figure~\ref{deg021} below. 
\end{remark}
	
\begin{corollary}\label{one_three}
Let $p(x,y)$ be a polynomial submersion with a convenient Newton polygon that has an outer edge $S$ with exactly one interior lattice point and let $c_0$ be the point of discontinuity of $N_S$. 
If $N_S(c_0) = 2$ and either (i) $N(c_0) = 3$ or (ii) there exists $\epsilon > 0$ such that the restriction of the function $N$ to $(c_0 - \epsilon, c_0)$ or to $(c_0, c_0 + \epsilon)$ is equal to $1$, then $p$ has no Jacobian mates. 
\end{corollary}
\begin{proof} 
Let $c_0$ as in Theorem~\ref{lattice_points}. 
The hyperbolic sectors of the assumptions of Lemma~\ref{connects} and contained in $p^{-1}(c_0)$ are provided by Theorem~\ref{lattice_points}. 
Both are in the same point of infinity $z$, and they satisfy, by Lemma~\ref{connects}, the sequence of hyperbolic sectors required by Theorem~\ref{hyperbolic}. 
The branch of an algebraic curve $\gamma$ given by Theorem~\ref{lattice_points} completes the set of assumptions required in Theorem~\ref{hyperbolic}. 
So $p$ has no Jacobian mates. 
\end{proof}

\begin{lemma}\label{new1}
Let $p(x,y)$ be a polynomial submersion with a convenient Newton polygon having an outer edge $S$ with exactly one interior lattice point. 
If the function $N-N_S$ is constant then $N \equiv 1$ or $p$ has no
Jacobian mates.  
\end{lemma}
\begin{proof} 
Let $A = N - N_S$. 
Then 
$$
-1=\int  N\,d\chi =\int (N-N_s)\,d\chi+\int N_s\,d\chi=-A+\int N_s\,d\chi. 
$$ 
It follows from Theorem~\ref{lattice_points} that $\int N_S\,d\chi \leq 0$. 
Therefore one of the following holds: 
\begin{enumerate}[label={\textnormal{(\roman*)}}]
\item\label{r1} $A=0$ and $\int N_s\,d\chi=-1$, 
\item\label{r2} $A=1$ and $\int N_s\,d\chi=0$. 
\end{enumerate}
According to Theorem~\ref{lattice_points}, if~\ref{r1} holds then $N_S$ is of type $1 {\bf 1} 1$ or $0 {\bf 1} 2$ or $2 {\bf 1} 0$ (here we are not assuming the positiveness of the coefficient of $x^{i_0} y^{j_0}$) and $N$ is of the same type. 
If $N$ is of type $1 {\bf 1} 1$ then it is constant. 
The other cases are impossible because the image of the real plane by $p$, which is an open mapping, would be a closed half interval. 

Now if~\ref{r2} holds then $N_S$ is of type $1 {\bf 2} 1$ or $0 {\bf 2} 2$ or $2 {\bf 2} 0$, that is, ${\bf b} = 2$ in Theorem~\ref{lattice_points}. 
Then $p$ is of type $2{\bf 3} 2$ or $1{\bf 3} 3$ or $3{\bf 3} 1$, respectively. 
Therefore, by Corollary~\ref{one_three}, it follows that $p$ has no Jacobian mates. 
\end{proof}

\begin{lemma}\label{lemma2}
Let $p(x,y)$ be a polynomial submersion whose Newton polygon has an outer edge $S$ with endpoints $(3,0)$ and $(3,3)$. 
Then
\begin{enumerate}[label={\textnormal{(\roman*)}}]
\item\label{i} $N_S$ is constant and bigger than $0$, or
\item\label{sii} $p$ does not have Jacobian mates, or
\item\label{siii} after an affine change of variables of the form $y:=y+a$, the Newton polygon of $p$ turns to one with an edge with endpoints $(3,3)$ and $(3,2)$ and one with endpoints $(3,2)$ and $(1,0)$. 
\end{enumerate}
\end{lemma}
See in Figure~\ref{3000} an example with the assumptions in (a) and the situation after~\ref{siii} in~(b). 
\begin{figure}[htbp]
\begin{center}
\subfigure[]{
\begin{tikzpicture}
\draw[-] (0.5,0)--(0.5,2.5);
\draw[-] (1,0)--(1,2.5);
\draw[-] (1.5,0)--(1.5,2.5);
\draw[-] (2,0)--(2,2.5);
\draw[-] (2.5,0)--(2.5,2.5);
\draw[-] (0,0.5)--(2.5,0.5);
\draw[-] (0,1)--(2.5,1);
\draw[-] (0,1.5)--(2.5,1.5);
\draw[-] (0,2)--(2.5,2);
\draw[-] (0,2.5)--(2.5,2.5);
\draw[->,thick](0,0)--(0,3);
\draw[->,thick](0,0)--(3,0);
\draw[-,red,thick](1.5,0)--(1.5,1.5);
\node at (1.7,0.75){{$S$}};
\draw[-,green,dashed,thick](1.5,1.5)--(1,2)--(.5,2.5);
\end{tikzpicture}
}
\quad \quad \quad \quad
\subfigure[]{
\begin{tikzpicture}
\draw[-] (0.5,0)--(0.5,2.5);
\draw[-] (1,0)--(1,2.5);
\draw[-] (1.5,0)--(1.5,2.5);
\draw[-] (2,0)--(2,2.5);
\draw[-] (2.5,0)--(2.5,2.5);
\draw[-] (0,0.5)--(2.5,0.5);
\draw[-] (0,1)--(2.5,1);
\draw[-] (0,1.5)--(2.5,1.5);
\draw[-] (0,2)--(2.5,2);
\draw[-] (0,2.5)--(2.5,2.5);
\draw[->,thick](0,0)--(0,3);
\draw[->,thick](0,0)--(3,0);
\draw[-,red,thick](0.5,0)--(1.5,1);
\draw[-,green,thick](1.5,1)--(1.5,1.5);
\draw[-,green,dashed,thick](1.5,1.5)--(1,2)--(0.5,2.5);
\end{tikzpicture}
}
\end{center}
\caption{The Newton polygons of Lemma~\ref{lemma2}.}\label{3000}
\end{figure}
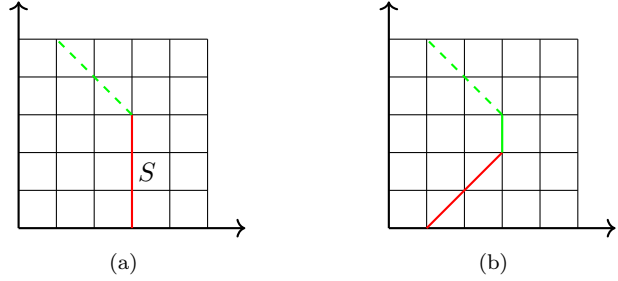 
\begin{proof}
If $p$ is nondegenerate on $S$ then we have Statement~\ref{i} from Lemma~\ref{bounded}. 

From now on we assume $p$ is degenerate on $S$, that is, $p_S(x,y) = \delta x^3 (y - a)^2 (y-b)$ for $a, b\in \R$. 
Then after substituting $y$ by $y+a$, it follows that $NP(p)$ reduces to one of the appearing in Figure~\ref{all} up to adding a suitable constant to $p$. 
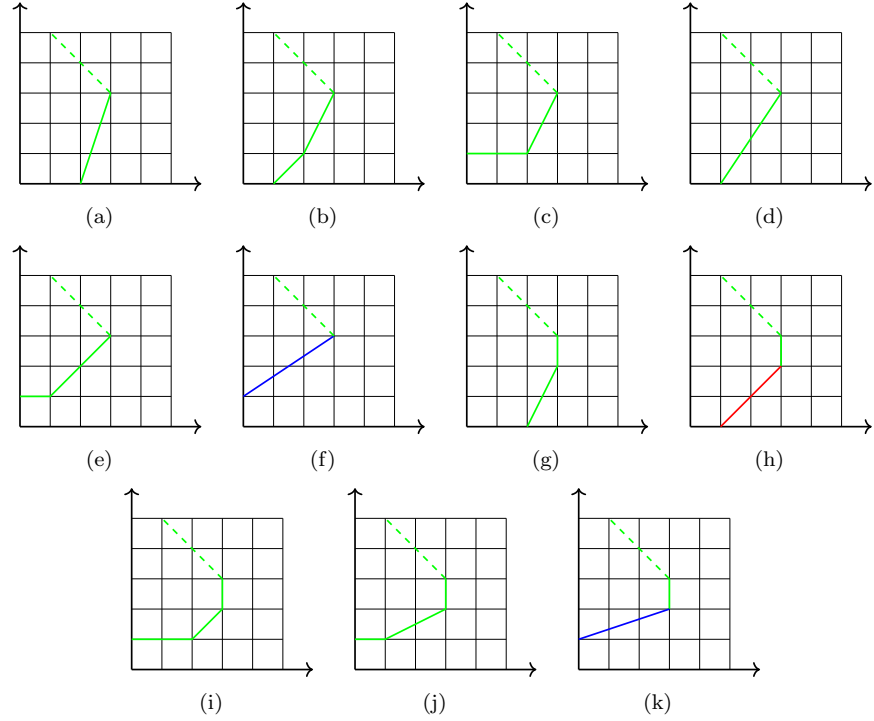
\begin{figure}[htbp]
\begin{center}
\newcommand{\tam}{0.8}
\subfigure[\label{aa}]{\scalebox{\tam}{
\begin{tikzpicture}
\draw[-] (0.5,0)--(0.5,2.5);
\draw[-] (1,0)--(1,2.5);
\draw[-] (1.5,0)--(1.5,2.5);
\draw[-] (2,0)--(2,2.5);
\draw[-] (2.5,0)--(2.5,2.5);
\draw[-] (0,0.5)--(2.5,0.5);
\draw[-] (0,1)--(2.5,1);
\draw[-] (0,1.5)--(2.5,1.5);
\draw[-] (0,2)--(2.5,2);
\draw[-] (0,2.5)--(2.5,2.5);
\draw[->,thick](0,0)--(0,3);
\draw[->,thick](0,0)--(3,0);
\draw[-,green,thick](1,0)--(1.5,1.5);
\draw[-,green,dashed,thick](1.5,1.5)--(1,2)--(.5,2.5);
\end{tikzpicture}}}\quad 
\subfigure[\label{ab}]{\scalebox{\tam}{
\begin{tikzpicture}
\draw[-] (0.5,0)--(0.5,2.5);
\draw[-] (1,0)--(1,2.5);
\draw[-] (1.5,0)--(1.5,2.5);
\draw[-] (2,0)--(2,2.5);
\draw[-] (2.5,0)--(2.5,2.5);
\draw[-] (0,0.5)--(2.5,0.5);
\draw[-] (0,1)--(2.5,1);
\draw[-] (0,1.5)--(2.5,1.5);
\draw[-] (0,2)--(2.5,2);
\draw[-] (0,2.5)--(2.5,2.5);
\draw[->,thick](0,0)--(0,3);
\draw[->,thick](0,0)--(3,0);
\draw[-,green,thick](0.5,0)--(1,0.5)--(1.5,1.5);
\draw[-,green,dashed,thick](1.5,1.5)--(1,2)--(.5,2.5);
\end{tikzpicture}}}\quad
\subfigure[\label{ac}]{\scalebox{\tam}{
\begin{tikzpicture}
\draw[-] (0.5,0)--(0.5,2.5);
\draw[-] (1,0)--(1,2.5);
\draw[-] (1.5,0)--(1.5,2.5);
\draw[-] (2,0)--(2,2.5);
\draw[-] (2.5,0)--(2.5,2.5);
\draw[-] (0,0.5)--(2.5,0.5);
\draw[-] (0,1)--(2.5,1);
\draw[-] (0,1.5)--(2.5,1.5);
\draw[-] (0,2)--(2.5,2);
\draw[-] (0,2.5)--(2.5,2.5);
\draw[->,thick](0,0)--(0,3);
\draw[->,thick](0,0)--(3,0);
\draw[-,green,thick](0,0.5)--(1,0.5)--(1.5,1.5);
\draw[-,green,dashed,thick](1.5,1.5)--(1,2)--(.5,2.5);
\end{tikzpicture}}}\quad 
\subfigure[\label{ad}]{\scalebox{\tam}{
\begin{tikzpicture}
\draw[-] (0.5,0)--(0.5,2.5);
\draw[-] (1,0)--(1,2.5);
\draw[-] (1.5,0)--(1.5,2.5);
\draw[-] (2,0)--(2,2.5);
\draw[-] (2.5,0)--(2.5,2.5);
\draw[-] (0,0.5)--(2.5,0.5);
\draw[-] (0,1)--(2.5,1);
\draw[-] (0,1.5)--(2.5,1.5);
\draw[-] (0,2)--(2.5,2);
\draw[-] (0,2.5)--(2.5,2.5);
\draw[->,thick](0,0)--(0,3);
\draw[->,thick](0,0)--(3,0);
\draw[-,green,thick](0.5,0)--(1.5,1.5);
\draw[-,green,dashed,thick](1.5,1.5)--(1,2)--(.5,2.5);
\end{tikzpicture}}}\quad
\subfigure[\label{ae}]{\scalebox{\tam}{
\begin{tikzpicture}
\draw[-] (0.5,0)--(0.5,2.5);
\draw[-] (1,0)--(1,2.5);
\draw[-] (1.5,0)--(1.5,2.5);
\draw[-] (2,0)--(2,2.5);
\draw[-] (2.5,0)--(2.5,2.5);
\draw[-] (0,0.5)--(2.5,0.5);
\draw[-] (0,1)--(2.5,1);
\draw[-] (0,1.5)--(2.5,1.5);
\draw[-] (0,2)--(2.5,2);
\draw[-] (0,2.5)--(2.5,2.5);
\draw[->,thick](0,0)--(0,3);
\draw[->,thick](0,0)--(3,0);
\draw[-,green,thick](0,0.5)--(0.5,0.5)--(1.5,1.5);
\draw[-,green,dashed,thick](1.5,1.5)--(1,2)--(.5,2.5);
\end{tikzpicture}}}\quad
\subfigure[\label{af}]{\scalebox{\tam}{
\begin{tikzpicture}
\draw[-] (0.5,0)--(0.5,2.5);
\draw[-] (1,0)--(1,2.5);
\draw[-] (1.5,0)--(1.5,2.5);
\draw[-] (2,0)--(2,2.5);
\draw[-] (2.5,0)--(2.5,2.5);
\draw[-] (0,0.5)--(2.5,0.5);
\draw[-] (0,1)--(2.5,1);
\draw[-] (0,1.5)--(2.5,1.5);
\draw[-] (0,2)--(2.5,2);
\draw[-] (0,2.5)--(2.5,2.5);
\draw[->,thick](0,0)--(0,3);
\draw[->,thick](0,0)--(3,0);
\draw[-,blue,thick](0,0.5)--(1.5,1.5);
\draw[-,green,dashed,thick](1.5,1.5)--(1,2)--(.5,2.5);
\end{tikzpicture}}}\quad
\subfigure[\label{ag}]{\scalebox{\tam}{
\begin{tikzpicture}
\draw[-] (0.5,0)--(0.5,2.5);
\draw[-] (1,0)--(1,2.5);
\draw[-] (1.5,0)--(1.5,2.5);
\draw[-] (2,0)--(2,2.5);
\draw[-] (2.5,0)--(2.5,2.5);
\draw[-] (0,0.5)--(2.5,0.5);
\draw[-] (0,1)--(2.5,1);
\draw[-] (0,1.5)--(2.5,1.5);
\draw[-] (0,2)--(2.5,2);
\draw[-] (0,2.5)--(2.5,2.5);
\draw[->,thick](0,0)--(0,3);
\draw[->,thick](0,0)--(3,0);
\draw[-,green,thick](1,0)--(1.5,1)--(1.5,1.5);
\draw[-,green,dashed,thick](1.5,1.5)--(1,2)--(.5,2.5);
\end{tikzpicture}}}\quad
\subfigure[\label{ah}]{\scalebox{\tam}{
\begin{tikzpicture}
\draw[-] (0.5,0)--(0.5,2.5);
\draw[-] (1,0)--(1,2.5);
\draw[-] (1.5,0)--(1.5,2.5);
\draw[-] (2,0)--(2,2.5);
\draw[-] (2.5,0)--(2.5,2.5);
\draw[-] (0,0.5)--(2.5,0.5);
\draw[-] (0,1)--(2.5,1);
\draw[-] (0,1.5)--(2.5,1.5);
\draw[-] (0,2)--(2.5,2);
\draw[-] (0,2.5)--(2.5,2.5);
\draw[->,thick](0,0)--(0,3);
\draw[->,thick](0,0)--(3,0);
\draw[-,green,thick](1.5,1)--(1.5,1.5);
\draw[-,red,thick](0.5,0)--(1.5,1);
\draw[-,green,dashed,thick](1.5,1.5)--(1,2)--(.5,2.5);
\end{tikzpicture}}}\quad
\subfigure[\label{ai}]{\scalebox{\tam}{
\begin{tikzpicture}
\draw[-] (0.5,0)--(0.5,2.5);
\draw[-] (1,0)--(1,2.5);
\draw[-] (1.5,0)--(1.5,2.5);
\draw[-] (2,0)--(2,2.5);
\draw[-] (2.5,0)--(2.5,2.5);
\draw[-] (0,0.5)--(2.5,0.5);
\draw[-] (0,1)--(2.5,1);
\draw[-] (0,1.5)--(2.5,1.5);
\draw[-] (0,2)--(2.5,2);
\draw[-] (0,2.5)--(2.5,2.5);
\draw[->,thick](0,0)--(0,3);
\draw[->,thick](0,0)--(3,0);
\draw[-,green,thick](0,0.5)--(1,0.5)--(1.5,1)--(1.5,1.5);
\draw[-,green,dashed,thick](1.5,1.5)--(1,2)--(.5,2.5);
\end{tikzpicture}}}\quad
\subfigure[\label{aj}]{\scalebox{\tam}{
\begin{tikzpicture}
\draw[-] (0.5,0)--(0.5,2.5);
\draw[-] (1,0)--(1,2.5);
\draw[-] (1.5,0)--(1.5,2.5);
\draw[-] (2,0)--(2,2.5);
\draw[-] (2.5,0)--(2.5,2.5);
\draw[-] (0,0.5)--(2.5,0.5);
\draw[-] (0,1)--(2.5,1);
\draw[-] (0,1.5)--(2.5,1.5);
\draw[-] (0,2)--(2.5,2);
\draw[-] (0,2.5)--(2.5,2.5);
\draw[->,thick](0,0)--(0,3);
\draw[->,thick](0,0)--(3,0);
\draw[-,green,thick](0,0.5)--(0.5,0.5)--(1.5,1)--(1.5,1.5);
\draw[-,green,dashed,thick](1.5,1.5)--(1,2)--(.5,2.5);
\end{tikzpicture}}}\quad
\subfigure[\label{ak}]{\scalebox{\tam}{
\begin{tikzpicture}
\draw[-] (0.5,0)--(0.5,2.5);
\draw[-] (1,0)--(1,2.5);
\draw[-] (1.5,0)--(1.5,2.5);
\draw[-] (2,0)--(2,2.5);
\draw[-] (2.5,0)--(2.5,2.5);
\draw[-] (0,0.5)--(2.5,0.5);
\draw[-] (0,1)--(2.5,1);
\draw[-] (0,1.5)--(2.5,1.5);
\draw[-] (0,2)--(2.5,2);
\draw[-] (0,2.5)--(2.5,2.5);
\draw[->,thick](0,0)--(0,3);
\draw[->,thick](0,0)--(3,0);
\draw[-,green,thick](1.5,1)--(1.5,1.5);
\draw[-,blue,thick](0,0.5)--(1.5,1);
\draw[-,green,dashed,thick](1.5,1.5)--(1,2)--(.5,2.5);
\end{tikzpicture}}}
\end{center}
\caption{The proof of Lemma~\ref{lemma2}.}\label{all}
\end{figure}

The cases (a), (b), (c), (d), (g), and (i) will provide Statement~\ref{i} because of the following remark: let $\tilde{p}$ be the polynomial $p$ after the application of the substitution $y \to y + a$, that is, $\tilde p(x,y) = p(x, y+a)$, and let $\gamma(t) = (x(t), y(t))$ be the parametrization of any branch of $\tilde p = c$ at infinity associated with an edge with outer normal vector $\xi = (\xi_1, \xi_2)$ with $\xi_1 \neq 0$, $\xi_2 < 0$. 
Applying the inverse of the substitution has the effect of adding the constant $-a$ to $y(t)$. 
Since $y(t) = B t^{-\xi_2} + \cdots$, it follows that after applying this inverse, we will get a branch of $p = c$ at infinity associated with $S$. 
Now, in all the cases (a), (b), (d) and (g), the edges with outer normal vector $\xi = (\xi_1, \xi_2)$ with $\xi_1 \neq 0$, $\xi_2 < 0$ are such that $\tilde p$ is nondegenerate on it, so we conclude that $N_S$ is constant. 
The cases (c) and (i) have the same reason up to the lower horizontal edge. 
This edge is horizontal for the polynomial $\tilde{p}(x,y) - 0$ and it turns to a nondegenerate edge with $\xi = (\xi_1,\xi_2)$, $\xi_1 \neq 0$, $\xi_2<0$, for $\tilde{p}(x,y) - c$, $c \neq 0$. 
So for all $c \neq 0$ we have a branch at infinity and for $c = 0$ we also have a branch at infinity coming from the factor $y$ of $\tilde{p}(x,y)$ has. 
In the first situation $y(t) = B t^{-\xi_2} + \cdots$ with $B \neq 0$, and in the second one, $y(t) \equiv 0$. 
But after adding $a$ to $y(t)$ we get a branch associated with $S$ in the original $p$. 

Cases (f) and (k) provide Statement~\ref{sii} by the main result of \cite{Gw2}. 
Case (h) gives Statement~\ref{siii}. 

Finally, in cases (e) and (j), we have $p(x,y) = a_{01} y + a_{11} x y + y^2 q(x,y)$ for suitable $q(x,y)$ and $a_{11} \neq 0$. 
Then $\nabla p(-a_{01}/a_{11},0) = (0,0)$, contradiction. 
\end{proof}

\section{The main result and its proof}\label{mr} 
\begin{theorem}\label{main}
Let $p:\R^2\to\R$ be a polynomial submersion function of degree $6$. 
If $B(p) \neq \emptyset$ then $p$ has no Jacobian mates. 
\end{theorem}
By this result and from the comments in the introduction section as well as in Section~\ref{453}, it follows that 
\begin{corollary}\label{cmain}
Let $F = (p,q): \R^2 \to \R^2$ be a polynomial map satisfying~\eqref{1}. 
If $\deg p < 7$, then $F$ is a global diffeomorphism. 
\end{corollary}

\begin{proof}[Proof of Theorem \ref{main}]
Since $p_+$ is homogeneous of degree $6$, it follows that 
$$
p_+ = q_1^{\alpha_1} q_2^{\alpha_2} q_3^{\alpha_3} h, 
$$ 
where $q_1$, $q_2$, $q_3$ are pairwise coprime linear forms, $\alpha_1, \alpha_2, \alpha_3 \in \{0,2,3,4,5,6\}$, $\alpha_1 \geq \alpha_2 \geq \alpha_3$, and $h$ has no multiple linear factors and is coprime with $q_1 q_2 q_3$. 
Applying a suitable linear change of coordinates we may assume that 
$$
q_1(x,y) = y, \quad q_2(x,y) = x,\quad q_3(x,y) = x+y. 
$$
In particular, $x$ and $y$ are not factors of $h$. 
By considering~\eqref{multiplefactor}, we can assume $\alpha_1 \geq 2$, so that $\deg h \leq 4$. 
Replacing $p(x,y)$ by $p(x - d,y - e) + f$ for suitable constants $d, e, f$, we may further assume that the Newton polygon of $p$ is convenient and does not have outer edges with outer normal vectors $(u,v)$ such that $u<0$ or $v<0$.

Below we analyze the $5$ possible cases $\deg h = 4$, $\deg h = 3$,...,$\deg h = 0$. 
In our analyzis, we will depict several Newton polygons, with outer edges drawing in three different colors: olive will be the color of the edge with normal $(1,1)$, that is, such that $p$ restricted to it is $p_+$, green will be the color of any other edge we know $p$ is nondegenerate on it, and red will be the color of an edge where $p$ can be degenerate. 
In particular, $p$ will always be nondegenerate on the olive edge when $\alpha_3 = 0$, what surely happens when $\deg h > 0$. 

\

\noindent 
\underline{If $\deg h = 4$.}
Then $p_+(x,y) = y^2 h(x,y)$ with $(4,0)$ and $(0,4)$ in the support of $h$ because $h$ is coprime with $x y$. 
It follows by our assumptions that $NP(p)$ is as in Figure~\ref{prima}. 
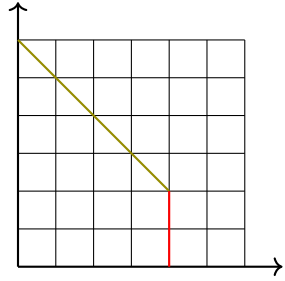
\begin{figure}[htbp]
\begin{center}
\begin{tikzpicture}
\draw[-] (0.5,0)--(0.5,3);
\draw[-] (1,0)--(1,3);
\draw[-] (1.5,0)--(1.5,3);
\draw[-] (2,0)--(2,3);
\draw[-] (2.5,0)--(2.5,3);
\draw[-] (3,0)--(3,3);
\draw[-] (0,0.5)--(3,0.5);
\draw[-] (0,1)--(3,1);
\draw[-] (0,1.5)--(3,1.5);
\draw[-] (0,2)--(3,2);
\draw[-] (0,2.5)--(3,2.5);
\draw[-] (0,3)--(3,3);
\draw[->,thick](0,0)--(0,3.5);
\draw[->,thick](0,0)--(3.5,0);
\draw[-,olive,thick](2,1)--(0,3);
\draw[-,red,thick](2,1)--(2,0);
\end{tikzpicture}
\end{center}
\caption{Newton polygon when $\deg h = 4$.}\label{prima} 
\end{figure}
Since $p$ is nondegenerate on the olive edge, it follows by Lemma~\ref{new1} that $p$ has no Jacobian mates.

\

\noindent 
\underline{If $\deg h =3$.} 
Then $p_+(x,y) = y^3 h(x,y)$ with $(3,0)$ and $(0,3)$ in the support of $h$ because $h$ is coprime with $x y$. 
It follows that $(5,0)$, $(4,1)$ and $(4,0)$ do not belong to $NP(p)$, otherwise $p$ is nondegenerate and hence it is a locally trivial fibration or is not a submersion, according to Remark~\ref{rem3}. 
So it remains to analyze when $NP(p)$ is as in (a) of Figure~\ref{3456}. 
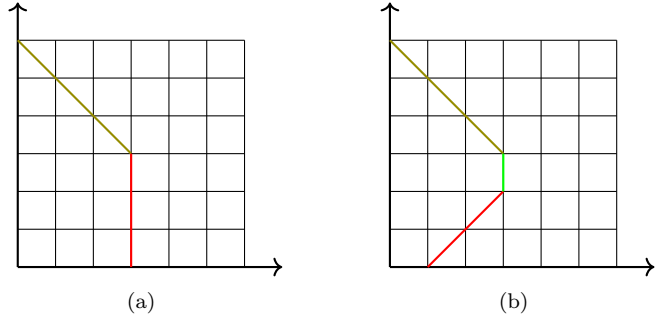
\begin{figure}[htbp]
\begin{center}
\subfigure[]{
\begin{tikzpicture}
\draw[-] (0.5,0)--(0.5,3);
\draw[-] (1,0)--(1,3);
\draw[-] (1.5,0)--(1.5,3);
\draw[-] (2,0)--(2,3);
\draw[-] (2.5,0)--(2.5,3);
\draw[-] (3,0)--(3,3);
\draw[-] (0,0.5)--(3,0.5);
\draw[-] (0,1)--(3,1);
\draw[-] (0,1.5)--(3,1.5);
\draw[-] (0,2)--(3,2);
\draw[-] (0,2.5)--(3,2.5);
\draw[-] (0,3)--(3,3);
\draw[->,thick](0,0)--(0,3.5);
\draw[->,thick](0,0)--(3.5,0);
\draw[-,olive,thick](1.5,1.5)--(0,3);
\draw[-,red,thick](1.5,1.5)--(1.5,0);
\end{tikzpicture}
}\quad \quad \quad
\subfigure[]{
\begin{tikzpicture}
\draw[-] (0.5,0)--(0.5,3);
\draw[-] (1,0)--(1,3);
\draw[-] (1.5,0)--(1.5,3);
\draw[-] (2,0)--(2,3);
\draw[-] (2.5,0)--(2.5,3);
\draw[-] (3,0)--(3,3);
\draw[-] (0,0.5)--(3,0.5);
\draw[-] (0,1)--(3,1);
\draw[-] (0,1.5)--(3,1.5);
\draw[-] (0,2)--(3,2);
\draw[-] (0,2.5)--(3,2.5);
\draw[-] (0,3)--(3,3);
\draw[->,thick](0,0)--(0,3.5);
\draw[->,thick](0,0)--(3.5,0);
\draw[-,red,thick](0.5,0)--(1.5,1);
\draw[-,olive,thick](1.5,1.5)--(0,3);
\draw[-,green,thick](1.5,1)--(1.5,1.5);
\end{tikzpicture}
}
\end{center}
\caption{Newton polygons when $\deg h = 3$.}\label{3456} 
\end{figure}
Here we use Lemma~\ref{lemma2}: in case~\ref{i}, since $p$ is nondegenerate on the olive edge, which has $4$ lattice points, it follows by Lemma~\ref{bounded} that $N(c)$ is constant bigger than $1$, so that $p$ cannot be a submersion. 
In case~\ref{sii}, $p$ has no Jacobian mates. 
Finally, in case~\ref{siii}, the Newton polygon of $p$ can be reduced to the one as in~(b) of Figure~\ref{3456}. 
Then we apply Lemma~\ref{new1} and get $p$ has no Jacobian mates. 

\

\noindent
\underline{If $\deg h =2$.} 
Then $p_+(x,y) = x^2 y^2h(x,y)$ or $p_+(x,y) = y^4h(x,y)$. 

\smallskip

\noindent 
\emph{Case $p_+(x,y) = x^2 y^2h(x,y)$:} since $(2,0)$ and $(0,2)$ belong to the support of $h$, it follows that up to interchanging $x$ and $y$, the possible \emph{degenerate} $NP(p)$ are the ones of Figure~\ref{455678}. 
\begin{figure}[htbp]
\begin{center}
\subfigure[]{
\begin{tikzpicture}
\draw[-] (0.5,0)--(0.5,3);
\draw[-] (1,0)--(1,3);
\draw[-] (1.5,0)--(1.5,3);
\draw[-] (2,0)--(2,3);
\draw[-] (2.5,0)--(2.5,3);
\draw[-] (3,0)--(3,3);
\draw[-] (0,0.5)--(3,0.5);
\draw[-] (0,1)--(3,1);
\draw[-] (0,1.5)--(3,1.5);
\draw[-] (0,2)--(3,2);
\draw[-] (0,2.5)--(3,2.5);
\draw[-] (0,3)--(3,3);
\draw[->,thick](0,0)--(0,3.5);
\draw[->,thick](0,0)--(3.5,0);
\draw[-,green,thick](0,2.5)--(1,2);
\draw[-,red,thick](2,1)--(2,0);
\draw[-,olive,thick](1,2)--(2,1);
\end{tikzpicture}
} \quad \quad \quad
\subfigure[]{
\begin{tikzpicture}
\draw[-] (0.5,0)--(0.5,3);
\draw[-] (1,0)--(1,3);
\draw[-] (1.5,0)--(1.5,3);
\draw[-] (2,0)--(2,3);
\draw[-] (2.5,0)--(2.5,3);
\draw[-] (3,0)--(3,3);
\draw[-] (0,0.5)--(3,0.5);
\draw[-] (0,1)--(3,1);
\draw[-] (0,1.5)--(3,1.5);
\draw[-] (0,2)--(3,2);
\draw[-] (0,2.5)--(3,2.5);
\draw[-] (0,3)--(3,3);
\draw[->,thick](0,0)--(0,3.5);
\draw[->,thick](0,0)--(3.5,0);
\draw[-,red,thick](0,2)--(1,2);
\draw[-,red,thick](2,1)--(2,0);
\draw[-,olive,thick](1,2)--(2,1);
\node at (.7, 2.2){$R$};
\node at (1.7, 1.7){$S$};
\node at (2.2, 0.74){$T$};
\end{tikzpicture}
}
\end{center}
\caption{When $\deg h = 3$ and $p_+=x^2 y^2h$.}\label{455678} 
\end{figure}
Since $\alpha_3 = 0$, it follows that $p$ is nondegenerate on the olive edge. 
In case~(a), we apply Lemma~\ref{new1} to conclude that $p$ has no Jacobian mates. 
In case~(b), let $R$, $S$ and $T$ denote the outer edges of $NP(p)$ ordered clockwise as in the figure. 
If $N_R$ or $N_T$ is constant then we apply Lemma~\ref{new1} once more to conclude that $p$ has no Jacobian mates. 
If both are nonconstant,  then Theorem~\ref{lattice_points} makes us conclude that $\int N_R\,d\chi\leq 0$ and  $\int N_T\,d\chi\leq 0$. 
Since $N_S$ is constant and equal to $0$ or $2$ by Lemma~\ref{bounded},  it follows from the equality 
$$
\int (N_R+N_S+N_T)\,d\chi=-1, 
$$
given by Corollary~\ref{corSek}, that, up to symmetry, 
$$
\int N_R\,d\chi = 0,\quad \int N_T\,d\chi =-1,\quad N_S \equiv 0.  
$$
The third equality implies that $h$ does not have real linear factors, so the signs of the coefficients $a_{24}$ and $a_{42}$ are the same, that we can assume positive up to multiplying $p$ by $-1$.  
Then by the two first equalities and Theorem~\ref{lattice_points} we obtain that $N_R$ is of type $0 {\bf 2} 2$ and $N_T$ is of type $0 {\bf 1} 2$. 
But this implies that the image of $p$ is a closed half-interval of the form $[c, \infty)$, with $c > -\infty$, contradiction because $p$ is a submersion. 

\smallskip

\noindent
\emph{Case $p_+(x,y) = y^4 h(x,y)$:} it is simple to see that $(5,0)$ and $(4,1)$ cannot belong to $NP(p)$, otherwise $p$ would be nondegenerate. 
So the possible $NP(p)$ are the ones of Figure~\ref{67252}.
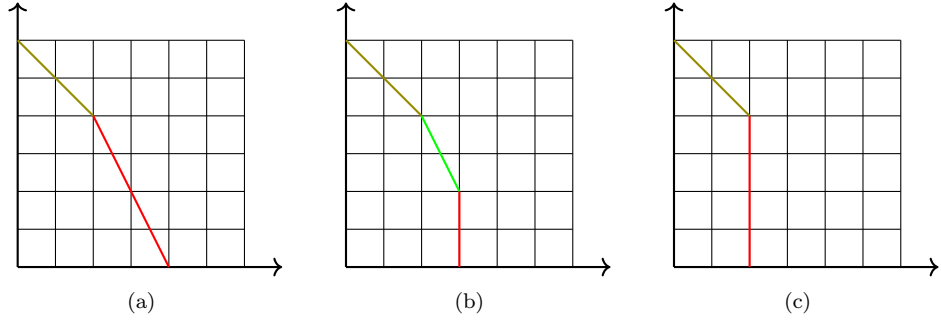
\begin{figure}[htbp]
\begin{center}
\subfigure[]{
\begin{tikzpicture}
\draw[-] (0.5,0)--(0.5,3);
\draw[-] (1,0)--(1,3);
\draw[-] (1.5,0)--(1.5,3);
\draw[-] (2,0)--(2,3);
\draw[-] (2.5,0)--(2.5,3);
\draw[-] (3,0)--(3,3);
\draw[-] (0,0.5)--(3,0.5);
\draw[-] (0,1)--(3,1);
\draw[-] (0,1.5)--(3,1.5);
\draw[-] (0,2)--(3,2);
\draw[-] (0,2.5)--(3,2.5);
\draw[-] (0,3)--(3,3);
\draw[->,thick](0,0)--(0,3.5);
\draw[->,thick](0,0)--(3.5,0);
\draw[-,red,thick](1,2)--(2,0);
\draw[-,olive,thick](0,3)--(1,2);
\end{tikzpicture}
} \quad 
\subfigure[]{
\begin{tikzpicture}
\draw[-] (0.5,0)--(0.5,3);
\draw[-] (1,0)--(1,3);
\draw[-] (1.5,0)--(1.5,3);
\draw[-] (2,0)--(2,3);
\draw[-] (2.5,0)--(2.5,3);
\draw[-] (3,0)--(3,3);
\draw[-] (0,0.5)--(3,0.5);
\draw[-] (0,1)--(3,1);
\draw[-] (0,1.5)--(3,1.5);
\draw[-] (0,2)--(3,2);
\draw[-] (0,2.5)--(3,2.5);
\draw[-] (0,3)--(3,3);
\draw[->,thick](0,0)--(0,3.5);
\draw[->,thick](0,0)--(3.5,0);
\draw[-,green,thick](1,2)--(1.5,1);
\draw[-,red,thick](1.5,1)--(1.5,0);
\draw[-,olive,thick](0,3)--(1,2);
\end{tikzpicture}
} \quad 
\subfigure[]{
\begin{tikzpicture}
\draw[-] (0.5,0)--(0.5,3);
\draw[-] (1,0)--(1,3);
\draw[-] (1.5,0)--(1.5,3);
\draw[-] (2,0)--(2,3);
\draw[-] (2.5,0)--(2.5,3);
\draw[-] (3,0)--(3,3);
\draw[-] (0,0.5)--(3,0.5);
\draw[-] (0,1)--(3,1);
\draw[-] (0,1.5)--(3,1.5);
\draw[-] (0,2)--(3,2);
\draw[-] (0,2.5)--(3,2.5);
\draw[-] (0,3)--(3,3);
\draw[->,thick](0,0)--(0,3.5);
\draw[->,thick](0,0)--(3.5,0);
\draw[-,red,thick](1,2)--(1,0);
\draw[-,olive,thick](0,3)--(1,2);
\end{tikzpicture}
}
\end{center}
\caption{When $\deg h = 3$ and $p_+ = y^4 h$.}\label{67252} 
\end{figure}
Since $p$ is nondegenerate on the olive edge, it follows by Lemma~\ref{new1} that $p$ has no Jacobian mates in cases (a) and (b). 
In case (c), $p$ is quadratic-like, so it has no Jacobian mates as well, according to Section~\ref{ql}. 

\

\noindent
\underline{If $\deg h =1$.} 
Then $p_+(x,y) = y^5h(x,y)$ or $p_+(x,y) = x^2y^3h(x,y)$.

\smallskip

\noindent
\emph{Case $p_+(x,y) = y^5 h(x,y)$:} it follows that $(5,0)$, $(4,1)$ and $(4,0)$ are not in $NP(p)$, otherwise $p$ is nondegenerate. 
Then the possible $NP(p)$ with $p$ having degree greater than $2$ in $x$ are the ones in Figure~\ref{figdeg1}. 
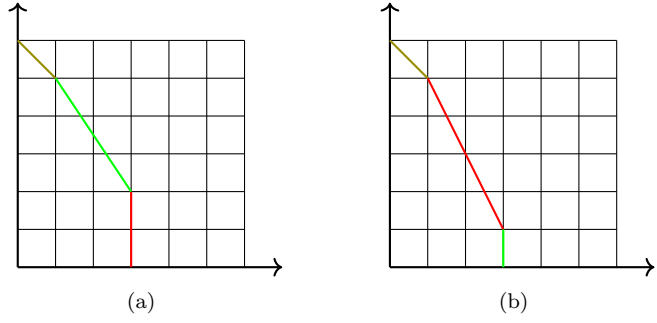
\begin{figure}[htbp]
\begin{center}
\subfigure[]{
\begin{tikzpicture}
\draw[-] (0.5,0)--(0.5,3);
\draw[-] (1,0)--(1,3);
\draw[-] (1.5,0)--(1.5,3);
\draw[-] (2,0)--(2,3);
\draw[-] (2.5,0)--(2.5,3);
\draw[-] (3,0)--(3,3);
\draw[-] (0,0.5)--(3,0.5);
\draw[-] (0,1)--(3,1);
\draw[-] (0,1.5)--(3,1.5);
\draw[-] (0,2)--(3,2);
\draw[-] (0,2.5)--(3,2.5);
\draw[-] (0,3)--(3,3);
\draw[->,thick](0,0)--(0,3.5);
\draw[->,thick](0,0)--(3.5,0);
\draw[-,red,thick](1.5,0)--(1.5,1);
\draw[-,green,thick](1.5,1)--(0.5,2.5);
\draw[-,olive,thick](0.5,2.5)--(0,3);
\end{tikzpicture}
}\quad \quad \quad
\subfigure[]{
\begin{tikzpicture}
\draw[-] (0.5,0)--(0.5,3);
\draw[-] (1,0)--(1,3);
\draw[-] (1.5,0)--(1.5,3);
\draw[-] (2,0)--(2,3);
\draw[-] (2.5,0)--(2.5,3);
\draw[-] (3,0)--(3,3);
\draw[-] (0,0.5)--(3,0.5);
\draw[-] (0,1)--(3,1);
\draw[-] (0,1.5)--(3,1.5);
\draw[-] (0,2)--(3,2);
\draw[-] (0,2.5)--(3,2.5);
\draw[-] (0,3)--(3,3);
\draw[->,thick](0,0)--(0,3.5);
\draw[->,thick](0,0)--(3.5,0);
\draw[-,red,thick](0.5,2.5)--(1.5,0.5);
\draw[-,green,thick](1.5,0.5)--(1.5,0);
\draw[-,olive,thick](0.5,2.5)--(0,3);
\end{tikzpicture}
}
\end{center}
\caption{When $\deg h = 1$ and $p_+(x,y) = y^5 h(x,y)$.}\label{figdeg1}
\end{figure}
In both cases, Lemma~\ref{new1} applies to conclude that $p$ has no Jacobian mates. 

\smallskip 

\noindent
\emph{Case $p_+(x,y) = x^2y^3 h(x,y)$:} it can happen that $(0,5)$ belongs or not to $NP(p)$. 
In the first case it follows that $(5,0), (4,1), (4,0) \notin NP(p)$, otherwise $p$ is nondegenerate. 
Then we get that $NP(p)$ is as in~(a) of Figure~\ref{sedeg1}. 
\begin{figure}[htbp]
\begin{center}
\subfigure[]{
\begin{tikzpicture}
\draw[-] (0.5,0)--(0.5,3);
\draw[-] (1,0)--(1,3);
\draw[-] (1.5,0)--(1.5,3);
\draw[-] (2,0)--(2,3);
\draw[-] (2.5,0)--(2.5,3);
\draw[-] (3,0)--(3,3);
\draw[-] (0,0.5)--(3,0.5);
\draw[-] (0,1)--(3,1);
\draw[-] (0,1.5)--(3,1.5);
\draw[-] (0,2)--(3,2);
\draw[-] (0,2.5)--(3,2.5);
\draw[-] (0,3)--(3,3);
\draw[->,thick](0,0)--(0,3.5);
\draw[->,thick](0,0)--(3.5,0);
\draw[-,green,thick](0,2.5)--(1,2);
\draw[-,olive,thick](1.5,1.5)--(1,2);
\draw[-,red,thick](1.5,1.5)--(1.5,0);
\end{tikzpicture}
}
\subfigure[]{
\begin{tikzpicture}
\draw[-] (0.5,0)--(0.5,3);
\draw[-] (1,0)--(1,3);
\draw[-] (1.5,0)--(1.5,3);
\draw[-] (2,0)--(2,3);
\draw[-] (2.5,0)--(2.5,3);
\draw[-] (3,0)--(3,3);
\draw[-] (0,0.5)--(3,0.5);
\draw[-] (0,1)--(3,1);
\draw[-] (0,1.5)--(3,1.5);
\draw[-] (0,2)--(3,2);
\draw[-] (0,2.5)--(3,2.5);
\draw[-] (0,3)--(3,3);
\draw[->,thick](0,0)--(0,3.5);
\draw[->,thick](0,0)--(3.5,0);
\draw[-,green,thick,dashed](1.5,1.5)--(2.5,0);
\draw[-,green,thick,dashed](1.5,1.5)--(2,0);
\draw[-,green,thick,dashed](1.5,1.5)--(2,0.5);
\draw[-,green,thick,dashed](2,0.5)--(2,0);
\draw[-,olive,thick](1.5,1.5)--(1,2);
\draw[-,red,thick](0,2)--(1,2);
\end{tikzpicture}
}
\subfigure[]{
\begin{tikzpicture}
\draw[-] (0.5,0)--(0.5,3);
\draw[-] (1,0)--(1,3);
\draw[-] (1.5,0)--(1.5,3);
\draw[-] (2,0)--(2,3);
\draw[-] (2.5,0)--(2.5,3);
\draw[-] (3,0)--(3,3);
\draw[-] (0,0.5)--(3,0.5);
\draw[-] (0,1)--(3,1);
\draw[-] (0,1.5)--(3,1.5);
\draw[-] (0,2)--(3,2);
\draw[-] (0,2.5)--(3,2.5);
\draw[-] (0,3)--(3,3);
\draw[->,thick](0,0)--(0,3.5);
\draw[->,thick](0,0)--(3.5,0);
\draw[-,red,thick](1.5,1.5)--(1.5,0);
\draw[-,olive,thick](1.5,1.5)--(1,2);
\draw[-,red,thick](0,2)--(1,2);
\node at (0.7, 2.19){$R$};
\node at (1.39, 1.86){$S$};
\node at (1.68, .74){$T$};
\end{tikzpicture}
}
\end{center}
\caption{When $\deg h = 1$ and $p_+ = x^2 y^3 h$.}\label{sedeg1}
\end{figure}
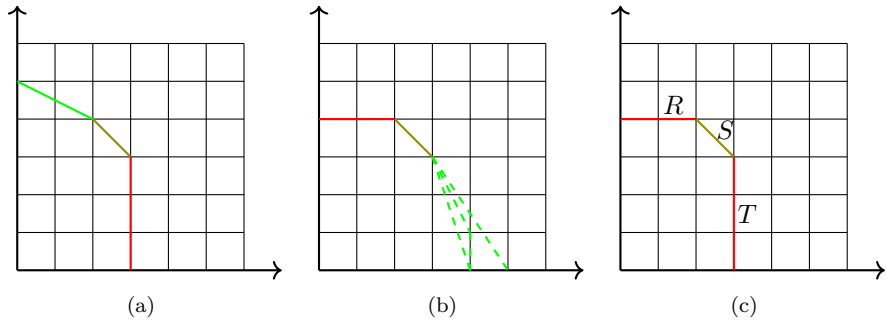
In the second case, that is $(0,5) \notin NP(p)$, it follows that the possible degenerate Newton polygons of $p$ are the three ones represented in~(b) or the one in~(c) of Figure~\ref{sedeg1}. 
In case~(a) of the figure, it follows by Lemma~\ref{lemma2} that either $p$ has no Jacobian mates or $NP(p)$ reduces to the one in~(a) of Figure~\ref{after3}. 
\begin{figure}[htbp]
\begin{center}
\subfigure[]{
\begin{tikzpicture}
\draw[-] (0.5,0)--(0.5,3);
\draw[-] (1,0)--(1,3);
\draw[-] (1.5,0)--(1.5,3);
\draw[-] (2,0)--(2,3);
\draw[-] (2.5,0)--(2.5,3);
\draw[-] (3,0)--(3,3);
\draw[-] (0,0.5)--(3,0.5);
\draw[-] (0,1)--(3,1);
\draw[-] (0,1.5)--(3,1.5);
\draw[-] (0,2)--(3,2);
\draw[-] (0,2.5)--(3,2.5);
\draw[-] (0,3)--(3,3);
\draw[->,thick](0,0)--(0,3.5);
\draw[->,thick](0,0)--(3.5,0);
\draw[-,green,thick](0,2.5)--(1,2);
\draw[-,green,thick](1.5,1.5)--(1.5,1);
\draw[-,olive,thick](1.5,1.5)--(1,2);
\draw[-,red,thick](1.5,1)--(0.5,0);
\end{tikzpicture}
}
\quad \quad \quad
\subfigure[]{
\begin{tikzpicture}
\draw[-] (0.5,0)--(0.5,3);
\draw[-] (1,0)--(1,3);
\draw[-] (1.5,0)--(1.5,3);
\draw[-] (2,0)--(2,3);
\draw[-] (2.5,0)--(2.5,3);
\draw[-] (3,0)--(3,3);
\draw[-] (0,0.5)--(3,0.5);
\draw[-] (0,1)--(3,1);
\draw[-] (0,1.5)--(3,1.5);
\draw[-] (0,2)--(3,2);
\draw[-] (0,2.5)--(3,2.5);
\draw[-] (0,3)--(3,3);
\draw[->,thick](0,0)--(0,3.5);
\draw[->,thick](0,0)--(3.5,0);
\draw[-,green,thick](1.5,1.5)--(1.5,1);
\draw[-,red,thick](1.5,1)--(0.5,0);
\draw[-,olive,thick](1.5,1.5)--(1,2);
\draw[-,red,thick](0,2)--(1,2);
\node at (0.7, 2.19){$R$};
\node at (1.39, 1.86){$S$};
\node at (1.68, 1.23){$T$};
\node at (1.19, .3){$U$};
\end{tikzpicture}
}
\end{center}
\caption{Reductions.}\label{after3}
\end{figure}
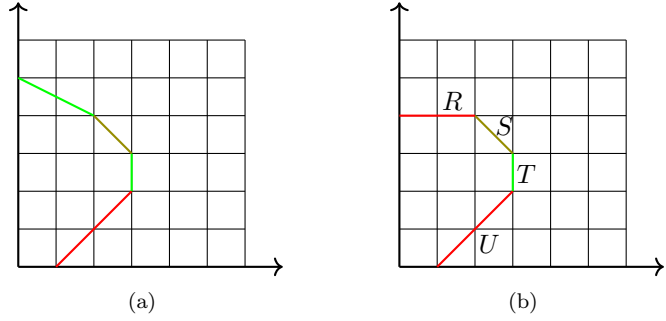
Then Lemma~\ref{new1} applies to show that $p$ has no Jacobian mates as well. 
In all the cases~(b) of Figure~\ref{sedeg1}, Lemma~\ref{new1} applies as well to conclude that $p$ has no Jacobian mates. 

We now analyze case~(c) of Figure~\ref{sedeg1}. 
If $N_T$ is constant, then Lemma~\ref{new1} allows us to conclude that $p$ has no Jacobian mates. 
If $N_T$ is not constant, then Lemma~\ref{lemma2} guarantees that $p$ has no Jacobian mates as well or $NP(p)$ reduces to the one in~(b) of Figure~\ref{after3}. 
Then we apply Lemma~\ref{bounded} and Theorem~\ref{lattice_points} to conclude that $\int \left(N_S + N_T\right)d\chi = -2$ and that $\int \left(N_R + N_U\right) d \chi \leq 0$, what contradicts Corollary~\ref{corSek}. 

\

\noindent 
\underline{If $\deg h =0$.} 
Then, after multyplying by a constant, $p_+(x,y)$,  is $y^6$, $x^2y^4$, $x^3y^3$ or $x^2y^2(x+y)^2$.

\smallskip

\noindent
\emph{Case $p_+(x,y) = y^6$:}
It follows that $(5,0), (4,1) \notin NP(p)$, otherwise $p$ is nondegenerate. 
Then, if $(3,2) \in NP(p)$, it follows further that $(4,0) \notin NP(p)$, otherwise $p$ is nondegenerate. 
So $NP(p)$ is as in~(a) of Figure~\ref{deg01}. 
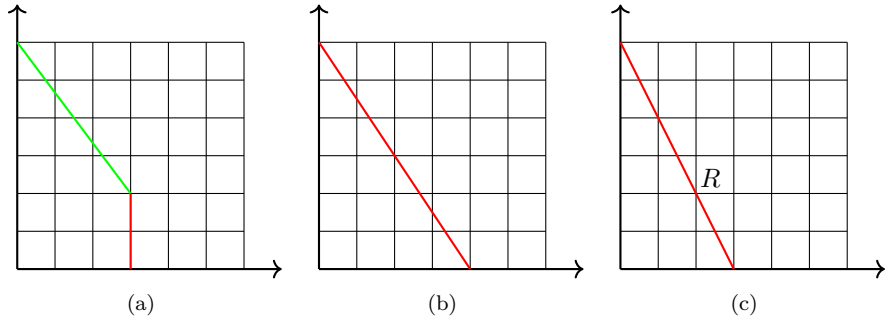
\begin{figure}[htbp]
\begin{center}
\subfigure[]{
\begin{tikzpicture}
\draw[-] (0.5,0)--(0.5,3);
\draw[-] (1,0)--(1,3);
\draw[-] (1.5,0)--(1.5,3);
\draw[-] (2,0)--(2,3);
\draw[-] (2.5,0)--(2.5,3);
\draw[-] (3,0)--(3,3);
\draw[-] (0,0.5)--(3,0.5);
\draw[-] (0,1)--(3,1);
\draw[-] (0,1.5)--(3,1.5);
\draw[-] (0,2)--(3,2);
\draw[-] (0,2.5)--(3,2.5);
\draw[-] (0,3)--(3,3);
\draw[->,thick](0,0)--(0,3.5);
\draw[->,thick](0,0)--(3.5,0);
\draw[-,green,thick](0,3)--(1.5,1);
\draw[-,red,thick](1.5,1)--(1.5,0);
\end{tikzpicture}
}
\subfigure[]{
\begin{tikzpicture}
\draw[-] (0.5,0)--(0.5,3);
\draw[-] (1,0)--(1,3);
\draw[-] (1.5,0)--(1.5,3);
\draw[-] (2,0)--(2,3);
\draw[-] (2.5,0)--(2.5,3);
\draw[-] (3,0)--(3,3);
\draw[-] (0,0.5)--(3,0.5);
\draw[-] (0,1)--(3,1);
\draw[-] (0,1.5)--(3,1.5);
\draw[-] (0,2)--(3,2);
\draw[-] (0,2.5)--(3,2.5);
\draw[-] (0,3)--(3,3);
\draw[->,thick](0,0)--(0,3.5);
\draw[->,thick](0,0)--(3.5,0);
\draw[-,red,thick](0,3)--(2,0);
\end{tikzpicture}
} 
\subfigure[]{
\begin{tikzpicture}
\draw[-] (0.5,0)--(0.5,3);
\draw[-] (1,0)--(1,3);
\draw[-] (1.5,0)--(1.5,3);
\draw[-] (2,0)--(2,3);
\draw[-] (2.5,0)--(2.5,3);
\draw[-] (3,0)--(3,3);
\draw[-] (0,0.5)--(3,0.5);
\draw[-] (0,1)--(3,1);
\draw[-] (0,1.5)--(3,1.5);
\draw[-] (0,2)--(3,2);
\draw[-] (0,2.5)--(3,2.5);
\draw[-] (0,3)--(3,3);
\draw[->,thick](0,0)--(0,3.5);
\draw[->,thick](0,0)--(3.5,0);
\draw[-,red,thick](0,3)--(1.5,0);
\node at (1.19, 1.2){$R$};
\end{tikzpicture}
}
\end{center}
\caption{When $\deg h = 0$ and $p_+ = y^6$.}\label{deg01}
\end{figure}
If, on the other hand, $(3,2) \notin NP(p)$, it follows that $p$ is nondegenerate, or  $NP(p)$ is as in~(b) or~(c) of Figure~\ref{deg01}, or $p$ is quadratic-like. 
In cases~(a) and~(b), Lemma~\ref{new1} applies to conclude that $p$ has no Jacobian mates. 
In case (c), we can assume that $p$ is nondegenerate, so $p_R(x,y) = (y^2 - c_1 x)^2 (y^2 - c_2 x)$, with suitable $c_1, c_2 \in \R$, $c_1 \neq 0$. 
Then the automorphism that replaces $y^2 - c_1 x$ by $x$ reduces the degree of $p$. 
And since the real Jacobian conjecture holds when the first coordinate of the map has degree less than $6$ by the main result of \cite{BFOO} cited above, it follows that $p$ has no Jacobian mates. 
	
\smallskip
	
\noindent
\emph{Case $p_+(x,y) = x^2 y^4$:}
First we suppose that $(0,5) \in NP(p)$. 
Then $(5,0), (4,1) \notin NP(p)$, otherwise $p$ is nondegenerate. 
Therefore, since $p$ is degenerate and not quadratic-like, it follows that $NP(p)$ is one of the two appearing in Figure~\ref{deg02}. 
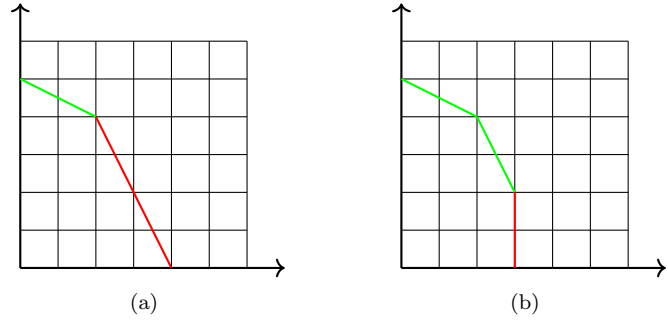
\begin{figure}[htbp]
\begin{center}
\subfigure[]{
\begin{tikzpicture}
\draw[-] (0.5,0)--(0.5,3);
\draw[-] (1,0)--(1,3);
\draw[-] (1.5,0)--(1.5,3);
\draw[-] (2,0)--(2,3);
\draw[-] (2.5,0)--(2.5,3);
\draw[-] (3,0)--(3,3);
\draw[-] (0,0.5)--(3,0.5);
\draw[-] (0,1)--(3,1);
\draw[-] (0,1.5)--(3,1.5);
\draw[-] (0,2)--(3,2);
\draw[-] (0,2.5)--(3,2.5);
\draw[-] (0,3)--(3,3);
\draw[->,thick](0,0)--(0,3.5);
\draw[->,thick](0,0)--(3.5,0);
\draw[-,green,thick](0,2.5)--(1,2);
\draw[-,red,thick](1,2)--(2,0);
\end{tikzpicture}
} \quad \quad \quad
\subfigure[]{
\begin{tikzpicture}
\draw[-] (0.5,0)--(0.5,3);
\draw[-] (1,0)--(1,3);
\draw[-] (1.5,0)--(1.5,3);
\draw[-] (2,0)--(2,3);
\draw[-] (2.5,0)--(2.5,3);
\draw[-] (3,0)--(3,3);
\draw[-] (0,0.5)--(3,0.5);
\draw[-] (0,1)--(3,1);
\draw[-] (0,1.5)--(3,1.5);
\draw[-] (0,2)--(3,2);
\draw[-] (0,2.5)--(3,2.5);
\draw[-] (0,3)--(3,3);
\draw[->,thick](0,0)--(0,3.5);
\draw[->,thick](0,0)--(3.5,0);
\draw[-,green,thick](0,2.5)--(1,2);
\draw[-,green,thick](1,2)--(1.5,1);
\draw[-,red,thick](1.5,1)--(1.5,0);
\end{tikzpicture}
}
\end{center}
\caption{When $\deg h = 0$, $p_+ = x^2 y^4$ and $(0,5) \in NP(p)$.}\label{deg02}
\end{figure}
Both cases have no Jacobian mates from Lemma~\ref{new1}.

Now we assume that $(0,5) \notin NP(p)$. 
Then the possible $NP(p)$ with $p$ not being quadratic-like are the ones of Figure~\ref{deg021}, 
\begin{figure}[htbp]
\begin{center}
\subfigure[]{
\begin{tikzpicture}
\draw[-] (0.5,0)--(0.5,3);
\draw[-] (1,0)--(1,3);
\draw[-] (1.5,0)--(1.5,3);
\draw[-] (2,0)--(2,3);
\draw[-] (2.5,0)--(2.5,3);
\draw[-] (3,0)--(3,3);
\draw[-] (0,0.5)--(3,0.5);
\draw[-] (0,1)--(3,1);
\draw[-] (0,1.5)--(3,1.5);
\draw[-] (0,2)--(3,2);
\draw[-] (0,2.5)--(3,2.5);
\draw[-] (0,3)--(3,3);
\draw[->,thick](0,0)--(0,3.5);
\draw[->,thick](0,0)--(3.5,0);
\draw[-,green,thick,dashed](1,2)--(2.5,0);
\draw[-,green,thick,dashed](1,2)--(2,0.5)--(2,0);
\draw[-,red,thick](0,2)--(1,2);
\end{tikzpicture}
}
\subfigure[]{
\begin{tikzpicture}
\draw[-] (0.5,0)--(0.5,3);
\draw[-] (1,0)--(1,3);
\draw[-] (1.5,0)--(1.5,3);
\draw[-] (2,0)--(2,3);
\draw[-] (2.5,0)--(2.5,3);
\draw[-] (3,0)--(3,3);
\draw[-] (0,0.5)--(3,0.5);
\draw[-] (0,1)--(3,1);
\draw[-] (0,1.5)--(3,1.5);
\draw[-] (0,2)--(3,2);
\draw[-] (0,2.5)--(3,2.5);
\draw[-] (0,3)--(3,3);
\draw[->,thick](0,0)--(0,3.5);
\draw[->,thick](0,0)--(3.5,0);
\draw[-,red,thick](0,2)--(1,2)--(2,0);
\node at (.74, 2.21){$R$};
\node at (1.67, 1.235){$S$};
\end{tikzpicture}
}
\subfigure[]{
\begin{tikzpicture}
\draw[-] (0.5,0)--(0.5,3);
\draw[-] (1,0)--(1,3);
\draw[-] (1.5,0)--(1.5,3);
\draw[-] (2,0)--(2,3);
\draw[-] (2.5,0)--(2.5,3);
\draw[-] (3,0)--(3,3);
\draw[-] (0,0.5)--(3,0.5);
\draw[-] (0,1)--(3,1);
\draw[-] (0,1.5)--(3,1.5);
\draw[-] (0,2)--(3,2);
\draw[-] (0,2.5)--(3,2.5);
\draw[-] (0,3)--(3,3);
\draw[->,thick](0,0)--(0,3.5);
\draw[->,thick](0,0)--(3.5,0);
\draw[-,red,thick](0,2)--(1,2);
\draw[-,red,thick](1.5,1)--(1.5,0);
\draw[-,green,thick](1,2)--(1.5,1);
\node at (.74, 2.21){$R$};
\node at (1.35, 1.75){$S$};
\node at (1.7, 0.74){$T$};
\end{tikzpicture}
}
\end{center}
\caption{When $\deg h = 0$, $p_+ = x^2 y^4$ and $(0,5) \not \in NP(p)$.}\label{deg021}
\end{figure}
where in (a) there are two possibilities drawn in the figure, both of them have no Jacobian mates by Lemma~\ref{new1}. 

As for the case~(b), we can assume that both $N_R$ and $N_S$ are not constant, otherwise we apply Lemma~\ref{new1} to conclude that $p$ has no Jacobian mates. 
Therefore, it follows from Theorem~\ref{lattice_points} that $N_R$ is of type $0{\bf b_1}2$ and $N_S$ is of type $0{\bf b_2}2$. 
By applying Corollary~\ref{corSek}, it follows that $b_1 + b_2 > 0$ so that the image of $p$ is a closed half-interval, a contradiction. 

Now we deal with the case~(c). 
From Lemma~\ref{new1}, we can assume that $N_R$ and $N_T$ are not constant. 
Also, Lemma~\ref{bounded} assures that $N_S \equiv 1$. 
Then it follows by Corollary~\ref{corSek}, after a glance in Theorem~\ref{lattice_points}, that $\int N_R d \chi = \int N_T d \chi = 0$,  so that $N_R$ is of type $0{\bf 2}2$ and $N_T$ is of type $1{\bf 2}1$. 
Observe that up to changing $x$ by $-x$ we can assume that the coefficient of $x^3 y^2$ is positive. 
This will in particular imply that for every $c\in \R$, the branch at infinity of $p^{-1}(c)$ associated with $S$ is contained in the plane $x < M$ for an $M\in \R$ big enough, according to~\eqref{branch} and Remark~\ref{pAB}. 
Let $c_R$ and $c_T$ be the points of discontinuity of the functions $N_R$ and $N_T$, respectively. 
We have three possibilities: 
\begin{enumerate}[label={\textnormal{(\roman*)}}]
\item\label{a1} $c_T < c_R$, 
\item\label{a3} $c_T = c_R$, 
\item\label{a2} $c_T > c_R$. 
\end{enumerate}
In possibility~\ref{a1}, it follows that $p$ is of type $2{\bf 3} 2 {\bf 4} 4$, then Corollary~\ref{one_three} applies to conclude that $p$ has no Jacobian mates. 
Now in possibilities~\ref{a3} and~\ref{a2}, we have that $p$ is of type $2{\bf 5} 4$ and $2{\bf 4} 4 {\bf 5} 4$, respectively. 
In order to conclude that $p$ has Jacobian mates, we will analyze the foliations in the Poincar\'e disc. 

We begin with possibility~\ref{a3}. 
We will assume in the analysis that $c_R = c_T = 0$. 
We shall apply Theorem~\ref{lattice_points} for both edges $R$ and $S$. 
First we consider the edge $R$. 
In this case $B = 4$. 
Then, by the theorem, see also Remark~\ref{localcharts}, there are four half branches of $p^{-1}(0)$ at infinity associated with the edge $R$ bounding two hyperbolic sectors, one in the north pole and one in the south pole of the Poincar\'e disc such that, inside them, the function assumes negative values. 
See~(a) of Figure~\ref{figurinha2}. 
\newcommand{\spa}{\hspace{-0.03cm}}
\begin{figure}[htbp]
\begin{center}
\psfrag{s1e}{\hspace{.07cm}\footnotesize $\ell$}
\psfrag{s2e}{\hspace{.1cm}\footnotesize $j$}
\psfrag{s2w}{\hspace{.15cm}\footnotesize $k$}
\subfigure[]{\includegraphics[scale=.25]{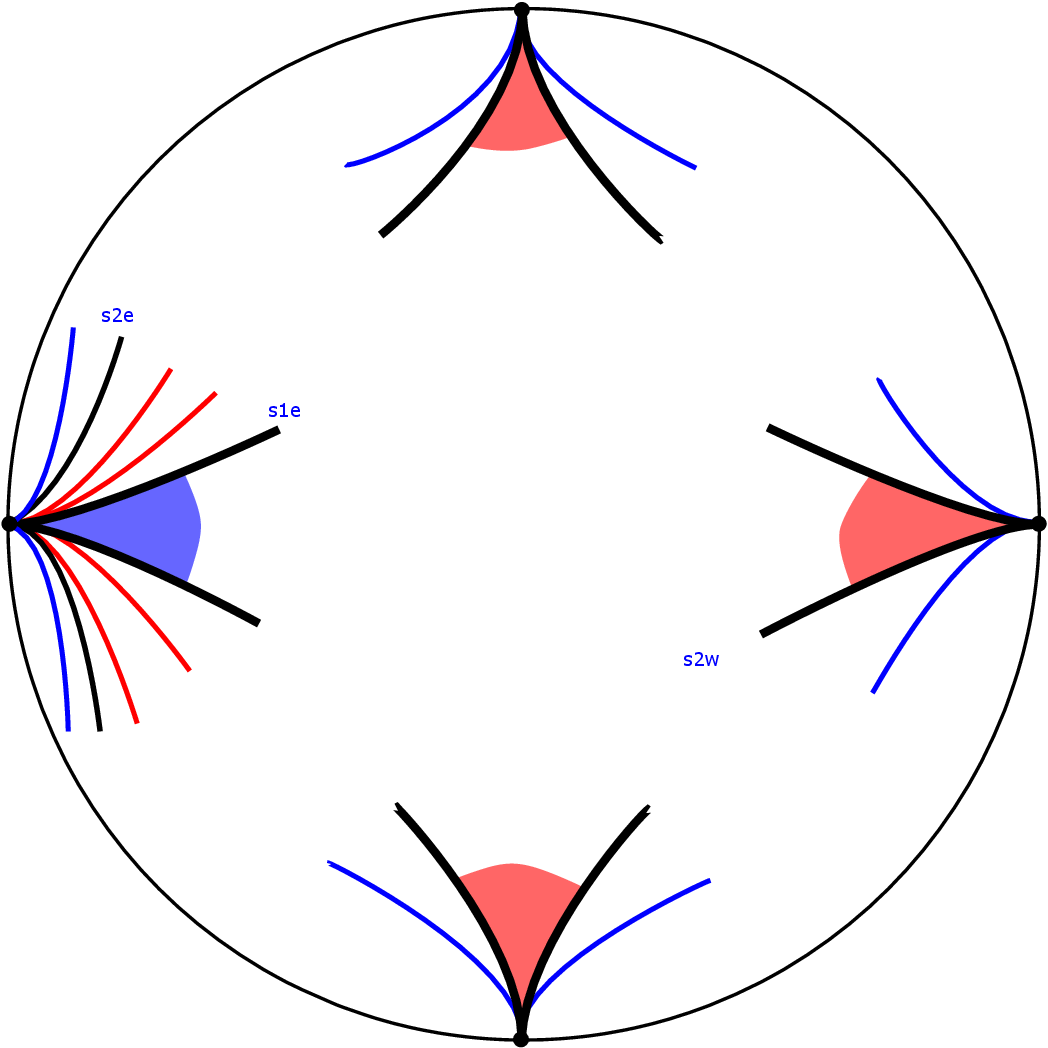}}\quad \quad \quad
\subfigure[]{\includegraphics[scale=.25]{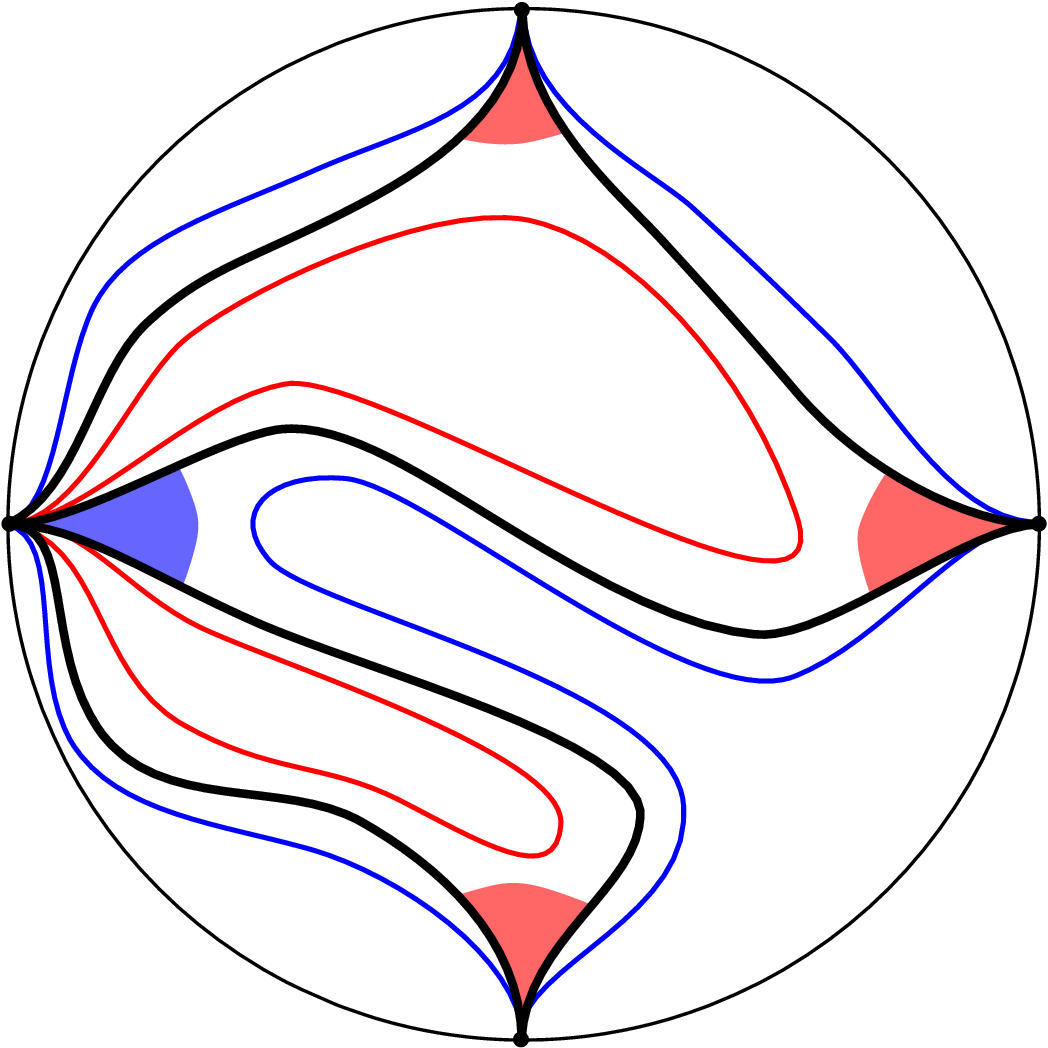}}
\subfigure[]{\includegraphics[scale=.25]{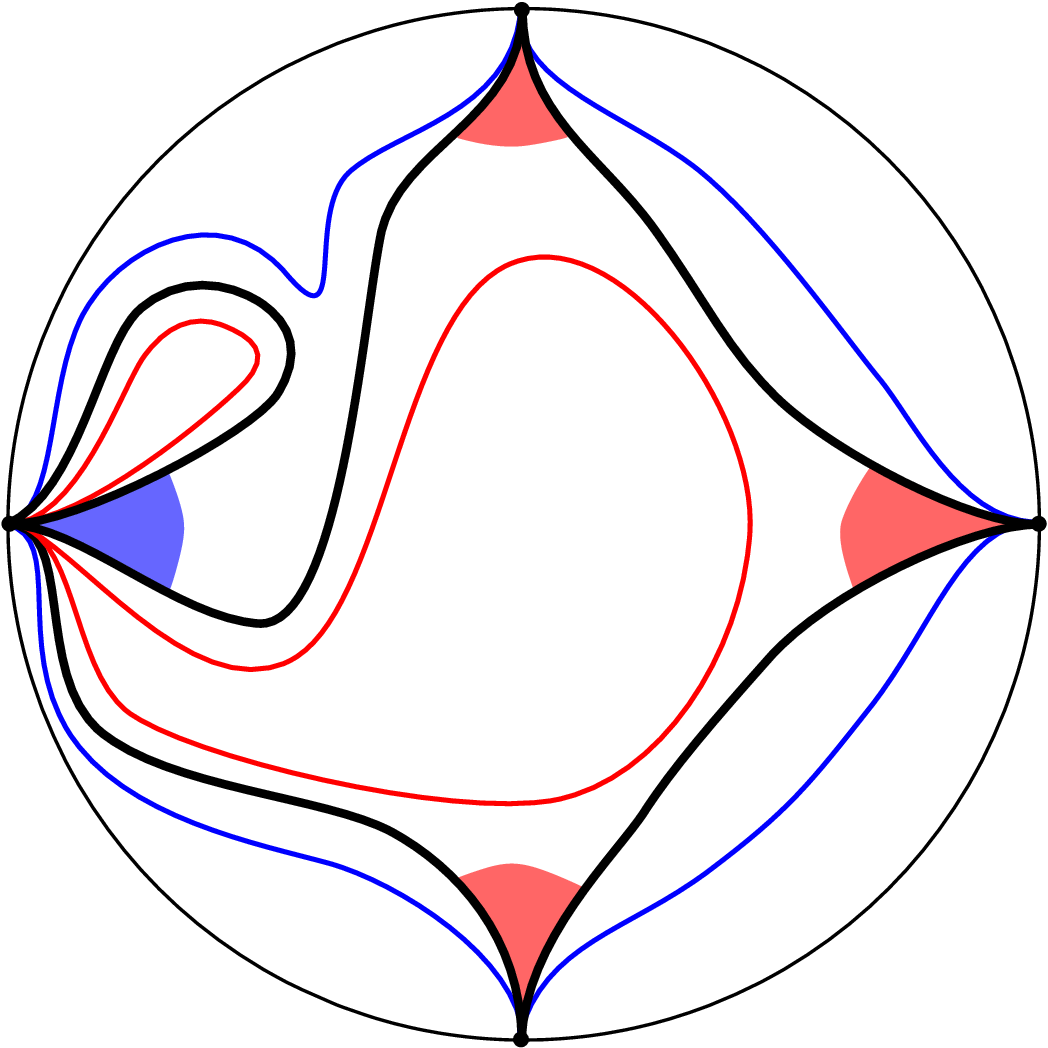}}
\end{center}
\caption{$c_R = c_T = 0$. Here we draw the values $c < 0$ in red, $c = 0$ in black, and $c > 0$ in blue. 
In (b) and (c) we draw two possible foliations coming from (a).}\label{figurinha2}
\end{figure}
We also draw the four half-branches of $p^{-1}(c)$, for $c > 0$, two in the north and two in the south. 
Now we repeat the process by considering the edge $T$. 
By Theorem~\ref{lattice_points} (as well as Remark~\ref{localcharts}), the four half-branches of $p^{-1}(0)$ at infinity associated with the edge $T$ bound two hyperbolic sectors, one in the west and one in the east, as in~(a) of Figure~\ref{figurinha2}. 
Since here $B = 3$, function $p$ assumes values $c < 0$ inside the hyperbolic sectors of the east, and $c > 0$ in the one of the west, again according to Theorem~\ref{lattice_points}. 
We also draw the $2$ half branches of $p^{-1}(c)$ associated with edge $T$ for $c< 0$ in the west and the two half branches of $p^{-1}(c)$ for $c > 0$ in the east. 
Finally, we will consider the edge $S$. 
As observed above, for every $c$, the branch of $p^{-1}(c)$ at infinity associated with $S$ will be in the west. 
In order to draw them, we will consider the parametrizations given by~\ref{branch}. 
For the branch $p^{-1}(c)$ at infinity associated with $S$, we have $\xi = (2,1)$. 
Then by Corollary~\ref{type}, it follows that the coordinate $y(t)$ of this parametrization has opposite sign and explodes to infinity when $t < 0$ and $t > 0$ are small enough. 
For the branches of $p^{-1}(0)$ associated with $T$, the border of the hyperbolic sectors in the west and east, we have $\xi = (1,0)$, so the coordinate $y(t)$ of their parametrizations tend to a finite value when $t \to 0$. 
This means that one of the half-branches of $p^{-1}(c)$ associated with $S$ is above the hyperbolic sector and one is below, as in the figure. 
Now we have to connect half-branches having the same colors. 
In order to avoid isolating either half-branches or hyperbolic sectors, and to match the colors inside the hyperbolic sectors, it is easy to conclude that up to symmetric behavior by means of changing $y$ to $-y$, there are two possibilities: the half-branch $\ell$ connects with $j$ or to $k$, see (a) of Figure \ref{figurinha2}. 
After these connections, we conclude the other connections as in (b) or (c) of Figure \ref{figurinha2}, respectively. 
In any of them, we can apply Theorem~\ref{hyperbolic} to conclude that $p$ has no Jacobian mates, because Theorem~\ref{lattice_points} guarantees a branch of an algebraic curve entering the each pair of diametricaly opposed hyperbolic sectors. 

Now in possibility~\ref{a2}, we assume, after changing $p$ by $(p - c_R)/(c_T-c_R)$, that $c_R = 0$ and $c_T = 1$. 
Theorem~\ref{lattice_points} can be applied for the edges $R$ and $T$ exactly as above to get the picture~(a) of Figure~\ref{figurinha} for these edges. 
\begin{figure}[htpb]
\begin{center}
\psfrag{s1e}{\phantom{a}}
\psfrag{s2e}{\hspace{.1cm}\footnotesize $k$}
\psfrag{s1w}{\hspace{.02cm}\footnotesize$\ell$}
\psfrag{s2w}{\phantom{a}}
\subfigure[]{\includegraphics[scale=.25]{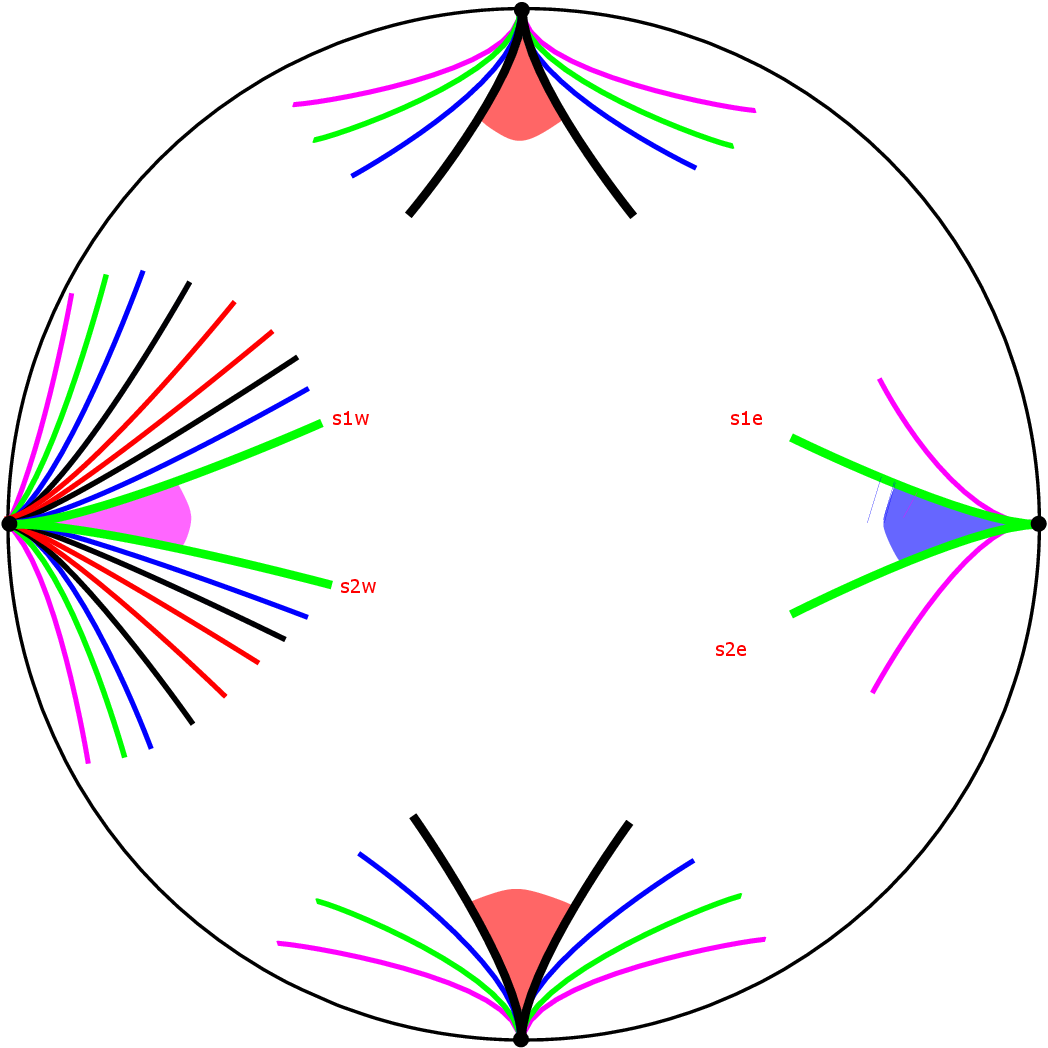}}\quad \quad \quad
\subfigure[]{\includegraphics[scale=.25]{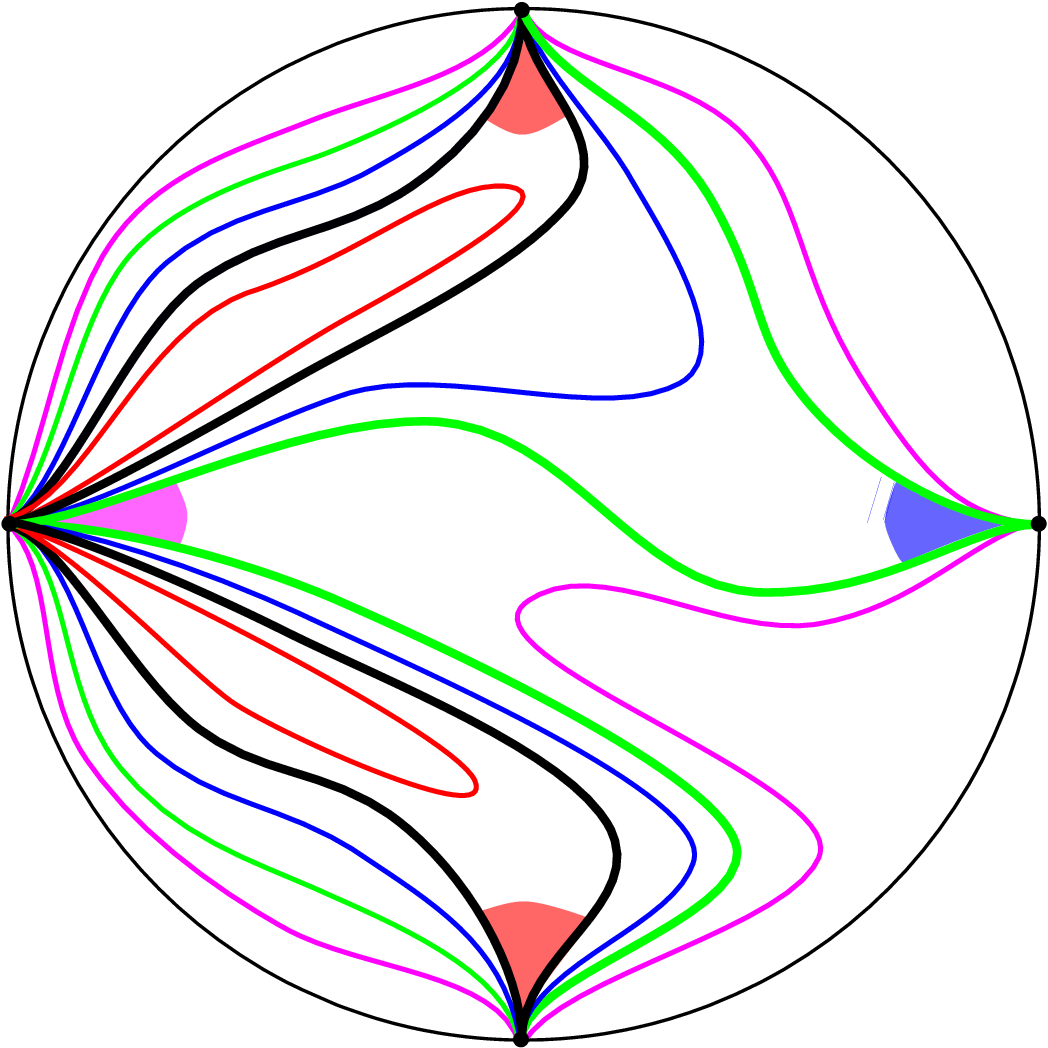}}
\end{center}
\caption{$0 = c_R < c_T = 1$. Here we draw the values $c<0$ in red, $c = 0$ in black, $0 < c < 1$ in blue, $c = 1$ in green, and $c>1$ in magenta.}\label{figurinha}
\end{figure}
Also for the edge $S$, we similarly use the expressions of the second coordinates of the parametrizations of the branches at infinity of $p^{-1}(c)$ associated with $S$ and of $p^{-1}(0)$ associated with $T$, to conclude that for each $c$, the branch of $p^{-1}(c)$ at infinity associated with $S$ has one half-branch above and one half-branch below the hyperbolic sector in the west, as in the figure. 
Here we also, see that, up to symmetry $y \to -y$, we need to connect $\ell$ to $k$. 
Then we finish the foliation as in (b) of Figure \ref{figurinha}. 
We apply Theorem~\ref{hyperbolic} to conclude that $p$ has no Jacobian mates, because Theorem~\ref{lattice_points} guarantees a branch of an algebraic curve entering the hyperbolic sectors in the west and in the east. 

\smallskip

\noindent
\emph{Case $p_+(x,y) = x^3y^3$:}
Up to exchanging $x$ and $y$, the degenerate $NP(p)$ are shown in Figure~\ref{deg03}, 
\begin{figure}[htbp]
\begin{center}
\subfigure[]{
\begin{tikzpicture}
\draw[-] (0.5,0)--(0.5,3);
\draw[-] (1,0)--(1,3);
\draw[-] (1.5,0)--(1.5,3);
\draw[-] (2,0)--(2,3);
\draw[-] (2.5,0)--(2.5,3);
\draw[-] (3,0)--(3,3);
\draw[-] (0,0.5)--(3,0.5);
\draw[-] (0,1)--(3,1);
\draw[-] (0,1.5)--(3,1.5);
\draw[-] (0,2)--(3,2);
\draw[-] (0,2.5)--(3,2.5);
\draw[-] (0,3)--(3,3);
\draw[->,thick](0,0)--(0,3.5);
\draw[->,thick](0,0)--(3.5,0);
\draw[-,red,thick](1.5,1.5)--(1.5,0);
\draw[-,green,dashed,thick](0,2)--(0.5,2)--(1.5,1.5);
\draw[-,green,dashed,thick](0,2)--(1.5,1.5);
\draw[-,green,dashed,thick](0,2.5)--(1.5,1.5);
\end{tikzpicture}
}\quad \quad \quad
\subfigure[]{ 
\begin{tikzpicture}
\draw[-] (0.5,0)--(0.5,3);
\draw[-] (1,0)--(1,3);
\draw[-] (1.5,0)--(1.5,3);
\draw[-] (2,0)--(2,3);
\draw[-] (2.5,0)--(2.5,3);
\draw[-] (3,0)--(3,3);
\draw[-] (0,0.5)--(3,0.5);
\draw[-] (0,1)--(3,1);
\draw[-] (0,1.5)--(3,1.5);
\draw[-] (0,2)--(3,2);
\draw[-] (0,2.5)--(3,2.5);
\draw[-] (0,3)--(3,3);
\draw[->,thick](0,0)--(0,3.5);
\draw[->,thick](0,0)--(3.5,0);
\draw[-,red,thick](0,1.5)--(1.5,1.5)--(1.5,0);
\node at (.73, 1.72){$R$};
\node at (1.7, .75){$S$};
\end{tikzpicture}
}
\end{center}
\caption{When $\deg h = 0$ and $p_+ = x^3 y^3$.}\label{deg03}
\end{figure}
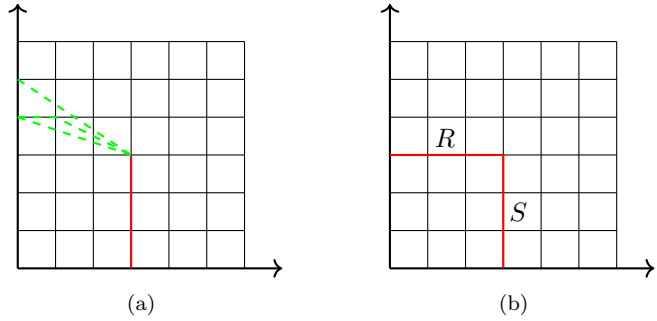
where, in case (a), there are three possibilities. 
For each of the possibilities of case (a), we apply lemmas~\ref{lemma2} and~\ref{new1} in sequence to conclude that $p$ has no Jacobian mates. 

Now we deal with case (b). 
If $N_R$ or $N_S$ is constant, then after application of lemmas~\ref{lemma2} and~\ref{new1} in sequence, we conclude that $p$ has no Jacobian mates. 
So we assume that both are not constant. 
Then after two further applications of Lemma~\ref{lemma2}, it follows that $p$ has no Jacobian mates, or $NP(p)$ can be reduced to the Newton polygon shown in Figure~\ref{reduced2}. 
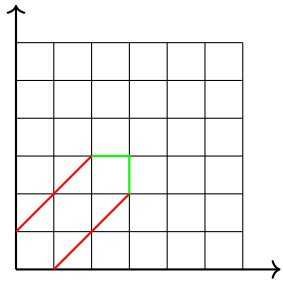
\begin{figure}[htbp]
\begin{center} 
\begin{tikzpicture}
\draw[-] (0.5,0)--(0.5,3);
\draw[-] (1,0)--(1,3);
\draw[-] (1.5,0)--(1.5,3);
\draw[-] (2,0)--(2,3);
\draw[-] (2.5,0)--(2.5,3);
\draw[-] (3,0)--(3,3);
\draw[-] (0,0.5)--(3,0.5);
\draw[-] (0,1)--(3,1);
\draw[-] (0,1.5)--(3,1.5);
\draw[-] (0,2)--(3,2);
\draw[-] (0,2.5)--(3,2.5);
\draw[-] (0,3)--(3,3);
\draw[->,thick](0,0)--(0,3.5);
\draw[->,thick](0,0)--(3.5,0);
\draw[-,red,thick](0,0.5)--(1,1.5);
\draw[-,red,thick](0.5,0)--(1.5,1);
\draw[-,green,thick](1.5,1)--(1.5,1.5)--(1,1.5);
\end{tikzpicture}
\end{center}
\caption{Reduction of case~(b) of Figure~\ref{deg03}.}\label{reduced2}
\end{figure}
Then it is simple to conclude that $\int N d\chi \leq -2$, a contradiction with Corollary~\ref{corSek}. 

\smallskip

\noindent
\emph{Case $p_+(x,y) = x^2 y^2 (x+y)^2$:} 
It is clear that, up to exchanging $x$ and $y$ we get the three possible $NP(p)$ of Figure~\ref{three}. 
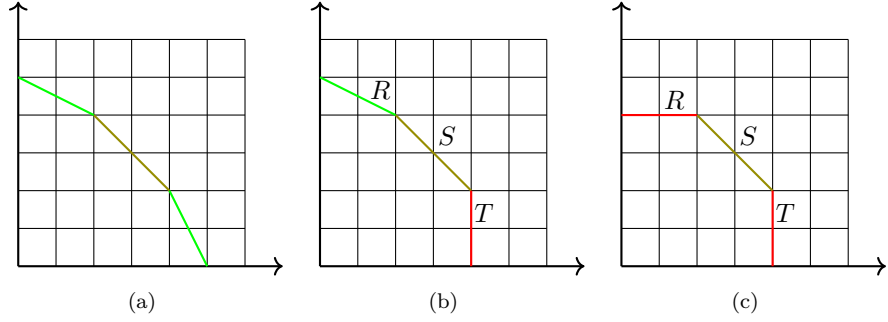
\begin{figure}[htbp]
\begin{center}
\subfigure[]{
\begin{tikzpicture}
\draw[-] (0.5,0)--(0.5,3);
\draw[-] (1,0)--(1,3);
\draw[-] (1.5,0)--(1.5,3);
\draw[-] (2,0)--(2,3);
\draw[-] (2.5,0)--(2.5,3);
\draw[-] (3,0)--(3,3);
\draw[-] (0,0.5)--(3,0.5);
\draw[-] (0,1)--(3,1);
\draw[-] (0,1.5)--(3,1.5);
\draw[-] (0,2)--(3,2);
\draw[-] (0,2.5)--(3,2.5);
\draw[-] (0,3)--(3,3);
\draw[->,thick](0,0)--(0,3.5);
\draw[->,thick](0,0)--(3.5,0);
\draw[-,olive,thick](1,2)--(2,1);
\draw[-,green,thick](0,2.5)--(1,2);
\draw[-,green,thick](2.5,0)--(2,1);
\end{tikzpicture}
}
\subfigure[]{
\begin{tikzpicture}
\draw[-] (0.5,0)--(0.5,3);
\draw[-] (1,0)--(1,3);
\draw[-] (1.5,0)--(1.5,3);
\draw[-] (2,0)--(2,3);
\draw[-] (2.5,0)--(2.5,3);
\draw[-] (3,0)--(3,3);
\draw[-] (0,0.5)--(3,0.5);
\draw[-] (0,1)--(3,1);
\draw[-] (0,1.5)--(3,1.5);
\draw[-] (0,2)--(3,2);
\draw[-] (0,2.5)--(3,2.5);
\draw[-] (0,3)--(3,3);
\draw[->,thick](0,0)--(0,3.5);
\draw[->,thick](0,0)--(3.5,0);
\draw[-,olive,thick](1,2)--(2,1);
\draw[-,red,thick](2,1)--(2,0);
\draw[-,green,thick](0,2.5)--(1,2);
\node at (.8, 2.33){$R$};
\node at (1.68, 1.72){$S$};
\node at (2.17, 0.7){$T$};
\end{tikzpicture}
} 
\subfigure[]{
\begin{tikzpicture}
\draw[-] (0.5,0)--(0.5,3);
\draw[-] (1,0)--(1,3);
\draw[-] (1.5,0)--(1.5,3);
\draw[-] (2,0)--(2,3);
\draw[-] (2.5,0)--(2.5,3);
\draw[-] (3,0)--(3,3);
\draw[-] (0,0.5)--(3,0.5);
\draw[-] (0,1)--(3,1);
\draw[-] (0,1.5)--(3,1.5);
\draw[-] (0,2)--(3,2);
\draw[-] (0,2.5)--(3,2.5);
\draw[-] (0,3)--(3,3);
\draw[->,thick](0,0)--(0,3.5);
\draw[->,thick](0,0)--(3.5,0);
\draw[-,olive,thick](1,2)--(2,1);
\draw[-,red,thick](0,2)--(1,2);
\draw[-,red,thick](2,1)--(2,0);
\node at (.7, 2.2){$R$};
\node at (1.68, 1.72){$S$};
\node at (2.17, 0.7){$T$};
\end{tikzpicture}
}
\end{center}
\caption{When $\deg h = 0$ and $p_+ = x^2 y^2 (x + y)^2$}\label{three}
\end{figure}
Here $p$ is degenerate on the olive edge. 
In case~(a), Lemma~\ref{new1} applies to show that $p$ has no Jacobian mates. 
In case~(b), by the same Lemma~\ref{new1}, we can assume that $N_S$ and $N_T$ are not constant, so by Theorem~\ref{lattice_points} it follows that $N_S$ is of type $0 {\bf b_S} 2$ and $N_T$ is of type $0 {\bf b_T} 2$, with $b_S, b_T \leq 2$. 
By Lemma~\ref{bounded}, $N_R \equiv 1$ so by Corollary~\ref{corSek} it follows that $\int N_S + N_T d\chi = 0$, forcing that $b_S = b_T = 2$. 
By letting $c_S$ and $c_T$ be the discontinuity points of $N_S$ and $N_T$, respectively, it follows that $p$ is of type $1{\bf 5}5$ if $c_S = c_T$, and of type $1{\bf 3} 3 {\bf 5} 5$ if $c_S \neq c_T$. 
In both cases, Corollary~\ref{one_three} applies to conclude that $p$ has no Jacobian mates. 

Now we deal with case~(c). 
By considering Lemma~\ref{new1} and up to symmetry, only three cases remain: (i)~$N_R$ is constant while $N_S$ and $N_T$ are not; (ii)~$N_S$ is constant while $N_R$ and $N_T$ are not; and (iii)~$N_R$, $N_S$, and $N_T$ are all not constant. 
In all the cases, $\int N_R, \int N_S, \int N_T \leq 0$ from Theorem~\ref{lattice_points}. 
Also, from the same theorem, when not constant, $N_i$ is of type $0{\bf b_i} 2$, with $b_i \leq 2$, for $i=R, S, T$. 
So, by considering Corollary~\ref{corSek} and observing that the image of $p$ cannot be a closed half-interval, it follows that, in case (i), we must have $N_R \equiv 1$ and $N_S$ and $N_T$ of type $0{\bf 2} 2$. 
Then $p$ is of type $1{\bf 5} 5$ or $1{\bf 3}3{\bf 5} 5$ depending on whether the discontinuity points of $N_S$ and $N_T$ are the same or not. 
In both cases, Corollary~\ref{one_three} applies to show that $p$ has no bifurcation values. 
A completely analogous reason can be done in case~(ii). 
Finally, in case~(iii), it follows that the image of $p$ is a closed half-interval because $b_R, b_S, b_T>0$. 
\end{proof}

\section{The example}\label{explicitizing}
We consider the polynomial submersion of degree $7$, $p_7$, presented in \cite{BF}. 
Here we solve one of the questions raised in that paper: we present a Jacobian mate for~$p_7$. 

Precisely, in our discussion below, we consider the translation 
$$
p = p_7 +1 = (1 + x + x^2 y) (1 + y + 2 x y + x^2 y^2). 
$$ 
It is easy to conclude that $p$ is a submersion such that $p^{-1}(0)$ is disconnected. 
So if $p$ has a Jacobian mate $q$, the map $(p,q)$ is non-injective by Section~\ref{45601}. 

We begin by considering the following auxiliary polynomials: 
$$
m = (1 + x + x^2 y),\quad w = (1 + x y) m. 
$$ 
Then we define 
$$
\begin{aligned}
q & = \frac{1}{m^2} \Big(2 (12 + 31 w) p^3 + (-6 - 68 w + 5 w^2 + 22 w^3) p^2 \\ 
& \phantom{={}} + 2 (7 + 11 w - 28 w^2 + 17 w^3 + w^4) w p  - (7 - 14 w + 7 w^2 + 11 w^4 - 10 w^5) w^2  \Big), 
\end{aligned}
$$ 
and 
$$
M_1 = \frac{-p + 5 p^2 - 6 p w^2 + 4 w^3 - 3 w^4}{m}, \quad M_2 = p (w^2 - p). 
$$
\begin{lemma}\label{qmm}
$q$, $M_1$ and $M_2$ are polynomials. 
\end{lemma}
\begin{proof}
$M_2$ is clearly polynomial. 
That $q$ and $M_1$ are polynomials follows because $m$ divides both $p$ and $w$. 
\end{proof}

\begin{lemma}
$M_1$ and $M_2$ have no common zeros. 
\end{lemma}
\begin{proof}
If $M_2 = 0$, we have (i) $p = w^2 \neq 0$ or (ii) $p = 0$. 

If (i) holds, then $m \neq 0$ and $m M_1 = - w^2 (2 w - 1)^2$. 
We \emph{claim that $p - w^2 = 0$ and $2 w - 1 = 0$ have no common zeros}. 
Indeed, by taking 
$$
r = 1 + 2 x y, \quad s = 1 + 2 x - 4 y + 2 x y + 8 x^2 y + 10 x^3 y^2 + 4 x^4 y^3, 
$$
we have 
$$
4 r (p - w^2) + s (2 w - 1) \equiv 1, 
$$
so the claim is proved. 

On the other hand, if (ii) holds, then either $m \neq 0$ or $m = 0$. 
If $m \neq 0$, we have 
\begin{equation}\label{eq}
y + (1 + x y)^2 = 1 + y + 2 x y + x^2 y^2 = 0. 
\end{equation} 
Moreover, $m M_1 =  w^3 (4 - 3 w)$, so that if $M_1 = 0$, then $w (4 - 3 w) = 0$. 
But $4 - 3 w = 1 - 3 x p/m= 1$, so $w = 0$, forcing that $1 + x y = 0$. 
Hence, by \eqref{eq}, $y = 0$ and we get a contradiction. 
Finally, if $m=0$, then $p = - m M_1$, and so $M_1$ can not be zero because $p$ is a submersion. 
\end{proof}

Next lemma follows by the rules of differentiation. 
\begin{lemma}\label{rty}
$$
J\Big(p,\frac{1}{m^2}\Big) = \frac{2 (2 w - 1)}{m^2}, \quad J(p,w) = - m p.
$$ 
\end{lemma}

\begin{lemma}
$$
J(p,q) = -2 \left(M_1^2 + 6 M_2^2\right)
$$
\end{lemma}
\begin{proof} 
By writing 
$$
q = \frac{1}{m^2} Q(p,w), 
$$ 
it follows from Lemma \ref{rty} that 
$$
J(p,q) = \frac{2 (2 w - 1) Q(p,w) - m p Q_w(p,w)}{m^2}. 
$$ 
Then, after a lengthy computation, we obtain 
\begin{equation}\label{equ}
J(p,q) + 2 \left(M_1^2 + 6 M_2^2 \right) =\frac{2 (m p-p + w- w^2) G(p,w,m)}{m^2}, 
\end{equation}
for a suitable polynomial $G(p,w,m)$. 

But  $m p=p-w+w^2$, by the expressions of $m$, $p$ and $w$. 
Therefore, since the left side of \eqref{equ} is polynomial by Lemma \ref{qmm}, the proof is complete. 
\end{proof}

Gathering the results of this section, we conclude that 
\begin{corollary}\label{ultimo} 
The map $F = (p,q): \R^2 \to \R^2$ is a non-injective polynomial map satisfying \ref{1}. 
The degree of $p$ is $7$. 
\end{corollary}

\section{Acknowledgments} 
FB was partially supported by the National Council for Scientific and Technological Development (CNPq), Brasil, grant 308112/2023-7. 
FB, FF and BOO were also partially supported by the National Council for Scientific and Technological Development (CNPq), Brasil, grant 403959/2023-3.

\end{document}